\documentclass{amsart}

\usepackage{graphicx}
\usepackage{hyperref}
\usepackage{amssymb}
\usepackage{pgfplots}
\pgfplotsset{compat=newest}

\numberwithin{equation}{section}
\newcommand{\be}{\begin{equation}}
\newcommand{\ee}{\end{equation}}
\newcommand{\HT}{{\mathcal{H}}_T}

\newcommand{\calS}{{\mathcal S}}

\newcommand{\diag}{\operatorname{diag}}
\newcommand{\spn}{\operatorname{span}}
\DeclareMathOperator*{\essinf}{ess\,inf}
\newtheorem{theorem}{Theorem}[section]

\newtheorem{lemma}[theorem]{Lemma}
\newtheorem{proposition}[theorem]{Proposition}
\newtheorem{corollary}[theorem]{Corollary}
\newtheorem{remark}[theorem]{Remark}

\newcommand{\cH}{\mathcal H}

\newcommand{\cL}{\mathcal L}

\newcommand{\IN}{\mathbb N}

\newcommand{\IR}{\mathbb R}

\newcommand{\domain}{\Omega}

\newcommand{\abs}[1]{\left\lvert{#1}\right\rvert}
\newcommand{\norm}[1]{{\left\lVert{#1}\right\rVert}}
\newcommand{\spf}[2]{{\left\langle{#1},{#2}\right\rangle}}

\definecolor{dgreen}{rgb}{0,.6,0}
\title[Wavelet Compression of $\HT$]{
  Wavelet compressed, modified Hilbert transform 
  in the space-time discretization of the heat equation}
\author{Helmut Harbrecht}
\address{H.~Harbrecht,
Department of Mathematics and Computer Science, 
University of Basel, Basel, Switzerland}
\email{helmut.harbrecht@unibas.ch}

\author{Christoph Schwab}
\address{Ch.~Schwab,
Seminar for Applied Mathematics, ETH Z\"urich, 
Z\"urich, Switzerland}
\email{schwab@math.ethz.ch}

\author{Marco Zank}
\address{M.~Zank,
Institute of Applied Mathematics, TU Graz, Graz, Austria}
\email{zank@math.tugraz.at}

\subjclass[2010]{Primary 65N30, 65J15}
\date{}
\dedicatory{}

\keywords{
Wavelets, Modified Hilbert transform, 
Space-time variational formulation of parabolic PDEs,
Generalized Hilbert transform,
Sparse space-time approximation of evolution
}

\begin{document}
\begin{abstract}
On a finite time interval $(0,T)$, we consider the 
multiresolution Galerkin discretization of a modified 
Hilbert transform $\cH_T$ which arises in the space-time 
Galerkin discretization of the linear diffusion equation. To 
this end, we design spline-wavelet systems in $(0,T)$ 
consisting of piecewise polynomials of degree $\geq 1$ 
with sufficiently many vanishing moments which constitute 
Riesz bases in the Sobolev spaces $ H^{s}_{0,}(0,T)$ and 
$ H^{s}_{,0}(0,T)$. These bases provide stable multilevel 
splittings of the temporal discretization spaces into ``increment'' 
or ``detail'' spaces. 
Furthermore, they allow to optimally 
compress the nonlocal integrodifferential operators which 
appear in stable space-time variational formulations of 
initial-boundary value problems, such as the heat 
equation and the acoustic wave equation.

We then obtain \emph{sparse space-time tensor-product 
spaces} via algebraic ten\-sor-products of the temporal multilevel
discretizations with standard, hierarchic finite element 
spaces in the spatial domain (with standard Lagrangian FE
bases). Hence, the construction of multiresolutions in 
the spatial domain is not necessary.  An efficient 
multilevel preconditioner is proposed that solves the 
linear system of equations resulting from the sparse 
space-time Galerkin discretization with essentially 
linear complexity (in work and memory). A substantial 
reduction in the number of the degrees of freedom and 
CPU time (compared to time-marching discretizations) 
is demonstrated in numerical experiments.
\end{abstract}
AMS Subject Classification: primary 65N30

\maketitle 

\section{Introduction}
\label{Sec:Intro}

\subsection{Motivation and background}
The efficient numerical solution of initial-boundary value problems (IBVPs) 
is central in computational science and engineering. Accordingly,
numerical methods have been developed to a high degree of sophistication
and maturity. 
Foremost among these are time-stepping schemes, which are motivated by
the causality of the physical phenomena modeled by the equations.
They discretize the evolution equation via sequential numerical solution
of a sequence of spatial problems \cite{Thomee2nd}.
In recent years, however, principally motivated by applications from
numerical optimal control, see \cite{Gunzburger} for example,
so-called \emph{space-time methods} have emerged: 
these methods aim at the ``one-shot'' solution of the initial-boundary
value problem as a well-posed operator equation on a space-time cylinder.
The present article develops an efficient space-time method for
linear, parabolic initial boundary value problems.

For linear parabolic evolution equations such as the heat equation,
the \emph{analytic semigroup} property will imply exponential
convergence rates with respect to the number of temporal degrees of
freedom if spectral or so-called $hp$ Petrov--Galerkin discretizations
are employed, see, e.g.\ \cite{PSZ23,SS2000}, 
\emph{provided} that the 
initial boundary value problem is subject to a forcing function 
which depends analytically on the temporal variable. 

In \emph{space-time Galerkin discretizations} developed in \cite{PSZ23}, 
the exponential convergence rate of the corresponding $hp$ time discretizations 
from \cite{SS2000} implies correspondingly small temporal stiffness and mass matrices, 
which can be efficiently treated by standard, dense linear algebra.

For \emph{linear hyperbolic evolution equations} 
such as acoustic waves or 
time-domain Maxwell equations (e.g.\ \cite{HauserZankMaxwell2024}),
temporal analyticity of solutions is not to be taken for granted. 
Similarly, 
for the aforementioned parabolic IBVPs subject to 
body forces with low temporal regularity (as, e.g., arise in 
pathwise solutions of linear, parabolic stochastic PDEs
driven by rough noises, such as cylindrical Wiener processes),
the solution has in turn low temporal regularity. 
In these settings, 
\emph{non-adaptive, low-order temporal discretization} 
are optimal, and the $hp$-approaches from \cite{PSZ23,SS2000} are not effective.
Due to the possibly dense singular support 
in the time variable of solutions, 
adaptive time-stepping will in such settings likewise 
not afford significant order improvements. 
This implies, in the presently considered temporal
Petrov--Galerkin formulation in a duality pairing 
which is realized by the temporal Hilbert transform, 
\emph{large, dense} stiffness 
and mass matrices of size $O(N_t\times N_t)$ with $N_t$ 
denoting the number of temporal degrees of freedom. 

A prototypical linear parabolic initial-boundary value problem 
considered in the present article is to find the function $u(x,t)$ 
such that
\begin{equation} \label{Einf:PDG}
    \left \{ \begin{array}{rclcl}
    \partial_t u  + \mathcal L_x u & = & f & \quad & \mbox{in} \;
     Q = \domain \times (0,1), \\[1mm]
    \gamma_{x,0}(u) & = & 0 & & \mbox{on} \; \Gamma_\mathrm{D} \times [0,1], 
    \\[1mm] 
    \gamma_{x,1}(u) & = & 0 & & \mbox{on} \; \Gamma_\mathrm{N} \times [0,1], 
    \\[1mm] 
    u(\cdot,0) & = & 0 & & \mbox{in} \; \domain
    \end{array} \right.
\end{equation}
with given right-hand side $f$ and, for simplicity at this stage, 
homogeneous Dirichlet boundary conditions on $\Gamma_\mathrm{D} 
\subset \partial \domain$, homogeneous Neumann boundary conditions 
on $\Gamma_\mathrm{N} \subset \partial \domain$ and 
homogeneous initial conditions. 
Here, $\domain \subset \IR^n$, $n \in \{1,2,3\}$, 
is a bounded domain, with Lipschitz boundary for $n\geq 2$, 
$\mathcal L_x$ is a spatial, linear and strongly elliptic differential operator, 
$\gamma_{x,0}$ denotes the spatial Dirichlet trace map, 
and $\gamma_{x,1}$ is the spatial conormal trace operator
on $\partial\domain$, see Subsection~\ref{sec:WellPos}
for further details.

\subsection{Previous results}
\label{sec:PrevRes}
As already mentioned, space-time discretizations are motivated by 
applications from optimal control, and also in the context of
space-time a-posteriori discretization error estimation.
These applications require access to the entire approximate 
solution over the entire time horizon $(0,T)$ of the IBVP.
Space-time discretizations of discontinuous Galerkin (dG) type also fall in this
rubric of time-marching schemes (\cite{PSZ23,SS2000,Thomee2nd}).
These $hp$-methods leverage \emph{exponential convergence rates}
of $hp$-time discretizations typically afforded by the \emph{analytic 
semigroup} property of the solution operator of the parabolic IBVP.

Although linear parabolic evolution equations such as \eqref{Einf:PDG}
are long known to be well-posed as operator equations in suitable Bochnerian
function spaces on the space-time cylinder $Q$ (e.g.\ \cite{DL92,ScStxtWav,ScStxtNSE}), 
the nonsymmetry and anisotropy of the parabolic operator $\partial_t + \mathcal L_x$
in \eqref{Einf:PDG}
have obstructed 
development of stable space-time variational discretizations. 
Progress towards stable space-time discretizations has been
made in \cite{ScStxtWav}, where the well-posedness of linear,
parabolic initial-boundary value problems in $Q$ as a well-posed 
operator-equation in suitable (in general fractional-order) Bochner--Sobolev
spaces in $Q$ has been leveraged in connection with adaptive wavelet methods
in $Q$. 
The stability of the continuous parabolic operator $\partial_t + \mathcal L_x$ 
and Riesz bases in $x$ and in $t$ 
in the corresponding function spaces were shown in \cite{ScStxtWav}
to afford stable, adaptive Galerkin algorithms, with certain optimality 
properties. 
These properties imply that these algorithms produce finite-parametric
approximations of the solution $u$ on $Q$ which converge at (essentially optimal) 
best $n$-term rates provided by suitable Besov-scale smoothness 
of solutions $u$ in $Q$. In addition, the Riesz basis properties of the 
multiresolution analyses in $Q$ enable optimal preconditioning of the resulting
finite-parametric approximations of the IBVP.
A wide range of linear and nonlinear parabolic IVBPs
has been shown to admit corresponding well-posedness results 
(e.g.\ \cite{ScStxtWav,ScStxtNSE} and the references there).
One obstruction to the wide applicability of such space-time adaptive wavelet schemes 
is the need to build spline-wavelet bases in possibly complicated, 
polytopal spatial domains $\domain$ with Riesz basis properties 
in a scale of Sobolev spaces in $\domain$. 
Efforts towards solving \eqref{Einf:PDG} as an operator equation 
without recourse to multiresolution bases comprise \cite{Andreev2013}.
There, unlike the approach in the present paper, a residual minimization
formulation has been adopted \cite[Eqn.~(3.10)]{Andreev2013}, which 
resulted in normal equation like linear systems.
These systems were to be solved numerically, and suitable preconditioners
were proposed in \cite{Andreev2013}. 
The concrete construction of bases satisfying the stable subspace-splitting
axioms in \cite{Andreev2013} 
still required Riesz-bases in the spatial domain \cite[Sec.~6.2.3]{Andreev2013}.
Recent effort has therefore been directed at developing
efficient numerical methods for compressive or sparse discretizations of 
space-time solvers of IBVPs, that obviate the construction of 
multiresolution analyses in the spatial domain.
We mention \cite{MR3449910,RStvWJW22} for linear parabolic IBVPs 
and \cite{BMPS21} for the acoustic wave equation in polygons.

\subsection{Contributions}
\label{sec:Contr}
We consider sparse space-time discretizations of 
IBVPs for \eqref{Einf:PDG}, 
with low-order discretizations in the temporal variable $t$ 
and the spatial variable $x$, with $N_t$ temporal degrees of
freedom in $t$ and $N_x$ degrees of freedom in the spatial domain.
We adopt a nonlocal, 
space-time variational formulation of the parabolic IBVP
which was recently proposed and investigated 
in \cite{SteinbachZankETNA2020,SteinbachZankJNUM2021}.
As in \cite{RStvWJW22}, we employ spline-wavelet bases for the temporal 
discretization and one-scale Galerkin finite elements in the spatial 
variable.
Distinct from formulations which are based on so-called 
least-squares formulations (e.g. \cite{GGRSt2021} and the references there),
we opt here for a near-symmetric variational formulation of $\partial_t$ 
in \eqref{Einf:PDG} and a standard, symmetric multilevel Galerkin discretization
of the spatial, elliptic operator in \eqref{Einf:PDG}.
The fractional-order temporal operator, 
comprising a modified Hilbert transform $\HT$
which arises in the discretization of the temporal variable, 
implies dense matrices with $N_t^2$ entries 
resulting from the temporal (Petrov--)Galerkin discretization 
with $N_t$ many degrees of freedom.
Leveraging the (by now classical) theory of
\emph{wavelet compression of pseudodifferential operators}
developed in the 90ies in e.g.\ \cite{DHS1,SchneiderBuch1998} and the references there,
we prove here that the matrices of the discrete, nonlocal evolution operators can, 
in suitable, so-called \emph{spline-wavelet bases} resp.\ multiresolution analyses be
\emph{compressed to $O(N_t)$ nonvanishing entries without loss of consistency orders}.

While $\HT$ is not a classical pseudodifferential
operator, we prove that, nevertheless, in suitable temporal 
spline-wavelet bases subject to causal boundary conditions which we construct of any fixed, 
given polynomial order, with a corresponding number of vanishing moments, will imply 
\emph{optimal compressibility of the temporal stiffness and mass matrices} 
of Petrov--Galerkin discretizations in these bases.
To this end, we verify here that derivatives of 
the distributional kernel $K(s,t)$ of $\cH_T$ satisfy
so-called Calder\'{o}n--Zygmund estimates, which are at the heart of 
wavelet compression analysis of classical pseudodifferential operators
in, e.g., \cite{DHS1} and the references there.
Combined with local analyticity estimates of $K(s,t)$ and the 
corresponding (exponentially convergent) quadrature techniques 
from \cite{ZankIntegral2023}, we develop here efficient, fully discrete 
algorithms for the $O(N_t)$ computation, storage and application of compressed 
Petrov--Galerkin discretizations of $\cH_T$.

The computational methodology proposed here is based on 
multiresolution Galer\-kin discretizations of $\cH_T$ 
in the time interval $I=(0,T)$ by biorthogonal spline-wavelets.
We construct concrete multiresolution analyses (MRAs)
for which most matrix entries in the resulting (dense)
Galerkin matrix are numerically negligible. 
I.e., 
they can be replaced by zero without compromising the overall accuracy 
of the Galerkin discretization. 
The number of relevant, nonzero matrix coefficients 
(which are located in a-priori known matrix entries)
scales linearly with $N_t$, the number of degrees of freedom,
while discretization error accuracy
of the underlying Galerkin scheme is retained.
Due to the piecewise polynomial structure of the MRA, the analyticity 
of the kernel function of $\cH_T$ can be leveraged in exponentially 
convergent numerical quadrature proposed in \cite{ZankExact2021}.
These findings apply even for the scalar case, i.e. for the 
initial value ODE, as we show in Section~\ref{sec:NumExp1d}.

Based on the sparsity pattern of the wavelet-compressed 
system matrix, a fill-in reducing reordering of the matrix 
entries by means of nested dissection is employed, 
see \cite{Geo73,LRT79}. 
As firstly proposed and demonstrated in \cite{HM21}, this 
reordering in turn allows for the rapid inversion of the system 
matrix by the Cholesky decomposition or more generally by 
a nested-dissection version of the LU decomposition. 
As no additional approximation errors 
are introduced, this is a major difference to other approaches 
for the discretization and the arithmetics of nonlocal operators, 
e.g.\ by means of hierarchical matrices. As the hierarchical 
matrix format is not closed under arithmetic operations, a 
recompression step after each arithmetic (block) operation 
has to be performed, which results in accumulating and 
hardly controllable consistency errors for matrix 
factorizations, see \cite{Hack1,Hack2}.

The efficient iterative solution of the (large) linear system 
of equations resulting from sparse tensor-product
space-time Galerkin discretizations
requires a suitable preconditioner. 
We apply here the standard 
BPX scheme in the spatial domain 
combined with sparse, direct inversion of the 
wavelet-compressed stiffness and mass matrices for the 
first-order, temporal derivative. 
That way, we arrive at an iterative 
solver which essentially requires a fixed number of iterations 
to achieve a prescribed accuracy.

Having available a hierarchical basis like wavelets 
in the temporal variable,
allows to build \emph{sparse tensor-product spaces
with any selection of standard finite elements in space}. 
Hence, we 
can apply sparse, space-time tensor-product spaces to 
space-time approximations of
the heat equation. 
From an asymptotic, error vs.\ accuracy
point of view, the temporal variable is then ``for free''. 
This corresponds to what $hp$ timestepping in \cite{SS2000,PSZ23} achieves,
but holds here in the absence of
analytic time-regularity of the solution.

We show how to modify the current 
implementation for the discretization with respect to full 
tensor-product space, using only ingredients that have
already been used for the full space-time approximation.
Our algorithm admits sparse space-time Galerkin 
approximation for linear parabolic evolution equations.
Unlike earlier works (e.g.\ \cite{PSZ23,ScStxtNSE} and 
the references there), the presently proposed discretization 
handles low spatial and temporal regularity, is non-adaptive, 
with the number of degrees of freedom, work and memory
scaling essentially as those for a multilevel solve of one 
stationary, elliptic boundary value problem. It allows a 
numerical solution of essentially optimal complexity in 
terms of the number of degrees of freedom. The fractional 
temporal order variational formulation is facilitated by a 
nonlocal, $H^{1/2}(0,T)$-type duality pairing with 
forward/backward causality, which is realized numerically 
by an optimal, $O(N_t)$ compression and preconditioning
based on spline-wavelets in the time-domain. 

We develop the theory and compression estimates for 
$T=1$. All results in the present article generalize, however,  
to arbitrary, finite time horizon $0<T<\infty$ by scaling.

\subsection{Layout}
\label{sec:Lyout}
The outline of this article is as follows. 
Section~\ref{Sec:Preliminaries} provides notation,
the temporal Sobolev spaces, and the properties of the 
modified Hilbert transform. In Section~\ref{Sec:FctSpcxt}, 
we then introduce the space-time variational formulation of 
the heat equation. The Galerkin discretization of this variational 
formulation in full tensor-product spaces is presented in 
Section~\ref{sec:Discr}. In order to arrive at a data-sparse 
representation of the temporal system matrices, biorthogonal 
spline-wavelets are defined in Section~\ref{sec:wavelets}, 
enabling the wavelet matrix compression of the temporal 
matrices defined in Section~\ref{sec:compression}. Respective 
numerical experiments are performed in Section~\ref{Sec:NumExp}. 
In Section~\ref{sec:SG}, we show that a sparse tensor-product
discretization is much more efficient than the full 
tensor-product approximation introduced before. Finally, 
concluding remarks are stated in Section~\ref{sec:conclusio}
while the coefficients of the wavelets used are given in 
the Appendix~\ref{sec:appendix}.

\section{Preliminaries}
\label{Sec:Preliminaries}
We start with introducing the notation used in the present article.
Then, we define fractional-order Sobolev spaces on the finite time 
interval. Without loss of generality, we assume that $T=1$, i.e.\ 
$I=(0,1)$. Finally, we present the properties of the modified Hilbert 
transform which will play a crucial role in our approach.

\subsection{Notation}
\label{sec:Notat}
We denote by $\IN = \{1,2,\dots\}$  the natural numbers, 
and set $\IN_0=\IN\cup \{0\}$. 
With the Lipschitz domain $\domain$, we shall denote a bounded, polytopal subset 
of the Euclidean space $\mathbb{R}^n$ where, mostly, $n=2$,
and with $I=(0,1)$ the time interval. 
With $Q = \domain \times I$, we shall denote the space-time
cylinder.
We shall use in various places tensor-products $\otimes$ of spaces. 
For finite-dimensional spaces, $\otimes$ shall always denote
the algebraic tensor-product. 
In the infinite-dimensional case, the symbol $\otimes$ shall
denote the Hilbertian (``weak'', or $w_2$) 
tensor-product of separable, real Hilbert spaces (see \cite{Ryan2002}).
For a countable set $\calS$, $|\calS|$ shall denote the number of elements
in $\calS$ whenever this number is finite. All function spaces are 
real-valued, and dualities are with respect to $\IR$ throughout.
We shall denote duality 
between Sobolev spaces in the temporal domain $I$ and the 
spatial domain $\domain$ by a $'$. 
We identify $L^2(I)\simeq L^2(I)'$ 
and 
$L^2(\domain) \simeq L^2(\domain)'$. 
Further, we 
denote by $\spf{\cdot}{\cdot}_I$ the duality pairing in $[H^{1/2}_{,0}(I)]'$ 
and $H^{1/2}_{,0}(I)$ as extension of the inner product in $L^2(I)$.
Lastly, for $s \in \IR$, $s >0$, the Hilbert spaces $H^s(I)$ are 
the usual Sobolev spaces, endowed with their usual norms 
$\| \cdot \|_{H^s(I)}$.

\subsection{Sobolev spaces on $I=(0,1)$}
\label{Sec:SobolevIntervall}
In this subsection, we recall the Sobolev spaces on intervals. 
The Lebesgue space $L^2(I)$ is endowed with the usual norm 
$\norm{\cdot}_{L^2(I)}$, whereas the Sobolev space $H^1(I)$ 
is equipped with the norm 
$\big( \norm{\cdot}_{L^2(I)}^2 + \norm{\partial_t \cdot}_{L^2(I)}^2 \big)^{1/2}$.  
Due to Poincaré inequalities, the subspaces
\[
  H^1_{0,}(I) = \{z \in H^1(I) |\ z(0) = 0 \}
\]
and
\[
  H^1_{,0}(I) = \{z \in H^1(I) |\ z(1) = 0 \}
\]
are endowed with the norm $\norm{\partial_t \cdot}_{L^2(I)}$. 
We also define fractional-order Sobolev spaces by interpolation 
\cite[Chapter~1]{LM1}, i.e. with the notation of \cite{LM1}
\[
  H^{s}_{0,}(I) = [H^1_{0,}(I), L^2(I)]_{s}
\]
and
\[
  H^{s}_{,0}(I) = [H^1_{,0}(I), L^2(I)]_{s}
\]
for $s \in [0,1]$. 

For $\ell \in \IN_0$,  we designate eigenfunctions $V_\ell(t) = \sqrt{2} 
\sin \left( \left( \frac{\pi}{2} + \ell \pi \right) t \right)$ and eigenvalues 
$\lambda_\ell = \frac{\pi^2}{4}(2\ell+1)^2$ 
of the eigenvalue problem
\[
  -\partial_{tt} V_\ell = \lambda_\ell V_\ell  \quad \text{ in } I, 
\quad V_\ell(0)=\partial_t V_\ell(1) = 0, \quad \norm{V_\ell}_{L^2(I)} = 1,
\]
see \cite[Subsection~3.4.1]{ZankDissBuch2020} for details. 
For the Sobolev space $H^{s}_{0,}(I)$, $s\in [0,1]$, 
we then consider the interpolation norm 
\[
\norm{z}_{H^{s}_{0,}(I)} 
= 
\left( \sum_{\ell=0}^{\infty} \lambda_\ell^s \abs{z_\ell}^2 \right)^{1/2}, 
\quad z \in H^{s}_{0,}(I),
\]
with the expansion coefficients $z_\ell = \int_0^1 z(t) V_\ell(t) \mathrm dt$. 
This norm coincides with the norm induced on $ H^{s}_{0,}(I) = 
[H^1_{0,}(I), L^2(I)]_{s}$ through the real method of interpolation 
(see \cite[Theorem 15.1 on page 98]{LM1}). 

Analogously, for $H^{s}_{,0}(I)$, $s\in [0,1]$, we consider the interpolation norm
\[
\norm{w}_{H^{s}_{,0}(I)} 
= 
\left( \sum_{\ell=0}^{\infty} \hat \lambda_\ell^s \abs{w_\ell}^2 \right)^{1/2}, 
\quad w \in H^{s}_{,0}(I),
\]
with the coefficients $w_\ell = \int_0^1 w(t) W_\ell(t) 
\mathrm dt$.
Here, $W_\ell(t) = \sqrt{2} 
\cos \left( \left( \frac{\pi}{2} + \ell \pi \right) t \right)$ 
and 
$\hat \lambda_\ell = \frac{\pi^2}{4}(2\ell+1)^2$ 
are eigenpairs of
\[
  -\partial_{tt} W_\ell = \hat \lambda_\ell W_\ell  \quad \text{ in } I, 
\quad \partial_t W_\ell(0)=  W_\ell(1) = 0, \quad \norm{W_\ell}_{L^2(I)} = 1.
\]
Note that $\lambda_\ell = \hat \lambda_\ell$ for all $\ell \in \IN_0$.

\subsection{Modified Hilbert transform}
\label{sec:HT}
We detail next the modified Hilbert transform $\HT$, which was
introduced in \cite{SteinbachZankETNA2020, ZankDissBuch2020}, 
and recall its main properties. It is an essential ingredient in 
a stable Petrov--Galerkin discretization of the 
variational formulation~\eqref{eq:IBVPVar} for the time derivative $\partial_t$. 
We only recap definitions and analytic properties. 
For details and proofs, we refer to 
\cite{LoescherSteinbachZankHT2024, SteinbachZankETNA2020, 
ZankDissBuch2020, ZankExact2021, ZankIntegral2023}.

For a given function $z \in L^2(I)$ with Fourier coefficients
\begin{equation*}
  z_k = \sqrt{2} \int_0^1 z(t) \sin \left( \left( \frac{\pi}{2} + k\pi \right) t \right) \mathrm dt = \int_0^1 z(t) V_k(t) \mathrm dt 
\end{equation*}
and the series representation
\begin{equation*}
  z(t) = \sum_{k=0}^\infty z_k
  \sqrt{2} \sin \left( \left( \frac{\pi}{2} + k\pi \right) t \right) = \sum_{k=0}^\infty z_k
  V_k(t), 
  \quad t \in I,
\end{equation*}
the modified Hilbert transform is defined by the series
\begin{equation}\label{HT:HT}
  (\HT z)(t) = \sum_{k=0}^\infty z_k \sqrt{2}
  \cos \left( \left( \frac{\pi}{2} + k\pi \right) t \right) = \sum_{k=0}^\infty z_k W_k(t), \quad t \in I.
\end{equation}
Here, $V_k$ and $W_k$ are eigenfunctions of corresponding 
eigenvalue problems given in Section~\ref{Sec:SobolevIntervall}.
For $s \in [0,1]$, the mapping
\begin{equation*}
    \HT \colon \, H^s_{0,}(I) \to H^s_{,0}(I)
\end{equation*}
is an isometry. 
This follows from
\[
  \norm{z}_{H^s_{0,}(I)}^2 
= \sum_{\ell=0}^{\infty} \lambda_\ell^s \abs{z_\ell}^2 
= \sum_{\ell=0}^{\infty} \hat \lambda_\ell^s \abs{z_\ell}^2 
= \norm{\HT z}_{H^s_{, 0}(I)}^2, \quad z \in H^s_{0,}(I).
\]
We also have 
\begin{equation} \label{HT:Beschr}
  \forall v,w \in H^{1/2}_{0,}(I) : 
\spf{\partial_t v}{\HT w}_I \leq \norm{v}_{H^{1/2}_{0,}(I)} \norm{w}_{H^{1/2}_{0,}(I)},
\end{equation}
see \cite[Eq.~(3.40)]{ZankDissBuch2020}. 
Further, the properties  
\begin{equation} \label{HT:Elliptisch}
    \forall v \in H^{1/2}_{0,}(I) : \spf{\partial_t v}{\HT v}_I = \norm{v}_{H^{1/2}_{0,}(I)}^2
\end{equation}
and
\begin{equation} \label{HTPositivsemi}
    \forall v \in L^2(I) : \spf{v}{\HT v}_{L^2(I)} \geq 0
\end{equation}
proven in \cite[Corollary~2.5]{SteinbachZankETNA2020} and 
\cite[Lemma~2.6]{SteinbachZankETNA2020} hold true. 
Adapting the proof of \cite[Lemma~2.6]{SteinbachZankETNA2020}, 
we even have
\begin{equation*}
    \forall s \in (0,1] : \forall\, 0\not= v \in H^s_{0,}(I): \spf{v}{\HT v}_{L^2(I)} > 0,
\end{equation*}
which is proven in \cite[Lemma~2.2]{LoescherSteinbachZankHT2024}.

In addition, we have that
\begin{equation*}
    \forall u, w \in L^2(I): \spf{\HT u}{w}_{L^2(I)} = \spf{u}{\HT^{-1}w}_{L^2(I)}
\end{equation*}
by \cite[Lemma~3.4.7]{ZankDissBuch2020} and
\begin{equation*}
    \forall u \in H^1_{0,}(I) : \forall v \in L^2(I): 
    \spf{\partial_t \HT u}{v}_{L^2(I)} = - \spf{\HT^{-1} \partial_t u}{v}_{L^2(I)}
\end{equation*}
by applying the same arguments as in the proof of 
\cite[Lemma~2.3]{SteinbachZankETNA2020}.

The modified Hilbert transform $\HT$ allows for integral 
representations, see \cite{SteinbachZankJNUM2021, 
ZankIntegral2023}, which is the motivation for the wavelet 
compression, originally introduced in the context of boundary 
integral equations. To present these, we recall that we considered 
without loss of generality $T=1$ (upon a dilation $t\mapsto t/T$). 

\begin{lemma}[Lemma~2.1 in \cite{SteinbachZankJNUM2021}] 
\label{Lem:HT:L2H1}
 For $v \in L^2(I),$ the modified Hilbert transform~$\HT$, defined in 
  \eqref{HT:HT}, allows the integral representation
  \begin{equation*}
    (\HT v)(t) = {\mathrm  {v.p.}} \int_0^1 K(s,t) \, v(x) \mathrm ds, \quad t \in I,
  \end{equation*}
  as a Cauchy principal value integral, where the kernel function is given as
  \begin{equation} \label{HT:L2Kern}
    K(s,t) = \frac{1}{2} \left[  \frac{1}{\sin \frac{\pi(s+t)}{2}  }  + \frac{1}{\sin \frac{\pi(s-t)}{2}  }   \right].
  \end{equation}
 For $v \in H^1(I),$ the operator $\HT$
 allows the integral representation
    \begin{equation*}
    (\HT v)(t) =
    - \frac{2}{\pi} v(0) \ln \tan \frac{\pi t}{4} + 
      \int_0^1 K_{-1}(s,t) \, \partial_s v(s) \mathrm ds, \quad t \in I,
    \end{equation*}
  as a weakly singular integral, where
  \begin{equation}\label{eq:K-1}
      K_{-1}(s,t) = -\frac{1}{\pi} \ln \left[ \tan \frac{\pi (s+t)}{4} \tan \frac{\pi \abs{t-s}}{4} \right].
  \end{equation}
\end{lemma}
\begin{remark}\label{rmk:HTTP}
We adopt the convention (with slight abuse of notation) 
that $\HT$ applied to functions in 
Bochner spaces $H^s_{0,}(I;H)$ taking values in a separable
Hilbert space $H$, i.e.\ $\HT\otimes \mathrm{Id}_H$, 
is denoted again by $\HT$.
\end{remark}

\section{Space-time variational formulation}
\label{Sec:FctSpcxt}
In this section, we present the space-time variational formulations 
upon which our Galerkin discretizations will be based. To this end,
we shall first introduce appropriate function spaces on the 
space-time cylinder $Q$.

\subsection{Bochner--Sobolev spaces on the space-time cylinder $Q$}
\label{sec:SobSpcxt}
Let $\domain \subset \IR^n$, $n\in\{1,2,3\}$, 
be a bounded, polytopal Lipschitz domain. 
We equip the Lebesgue space $L^2(\domain)$ 
with the usual norm $\norm{\cdot}_{L^2(\domain)}$ and the 
Sobolev space $H^1_0(\domain)$, fulfilling homogeneous 
Dirichlet conditions on $\partial \domain$, with the norm 
$\norm{\cdot}_{H^1_0(\domain)} = \abs{\cdot}_{H^1(\domain)} 
= \norm{\nabla_x \cdot}_{L^2(\domain)^n}$. More generally, 
we divide the boundary $\partial \domain$ into a Dirichlet 
boundary $\Gamma_\mathrm{D}$ and a Neumann boundary 
$\Gamma_\mathrm{N}$, where each boundary part consists 
of a union of $n_\mathrm{D} \in \IN$ resp.\ $n_\mathrm{N} 
\in \IN_0$ faces of the spatial domain $\domain$. Note that,
by assumption, we have $n_\mathrm{D}>0$, $n_\mathrm{N}
\geq 0$. Thus, $|\Gamma_\mathrm{D}|>0$ if $n=2,3$ or 
$\Gamma_\mathrm{D}\ne\emptyset$ if $n=1$, i.e.\ there is 
a nonempty Dirichlet part. With this notation, the space 
$H^1_{\Gamma_\mathrm{D}}(\domain)$ consists of all 
functions in the Sobolev space $H^1(\domain)$, fulfilling 
homogeneous Dirichlet conditions on $\Gamma_\mathrm{D}$, 
and is equipped with the norm $\norm{\cdot}_{H^1_{\Gamma_\mathrm{D}}
(\domain)} = \abs{\cdot}_{H^1(\domain)} = \norm{\nabla_x \cdot}_{L^2(\domain)^n}$. 
Further, for $s \geq 0$, the Hilbert spaces $H^s(\domain)$
are the usual Sobolev spaces endowed with their usual norms 
$\| \cdot \|_{H^s(\domain)}$.

For the space-time 
cylinder $Q=\domain \times I$, we extend the Sobolev
spaces on intervals of Subsection~\ref{Sec:SobolevIntervall} 
to vector-valued Bochner--Sobolev spaces, e.g., 
\begin{align*}
  L^2(I;H^1_{\Gamma_\mathrm{D}}(\domain)) &\simeq H^1_{\Gamma_\mathrm{D}}(\domain) \otimes L^2(I),
  \\
  H^s_{0,}(I;L^2(\domain)) &\simeq L^2(\domain) \otimes H^s_{0,}(I) \quad \text{ for } s \in [0,1],
  \\
  H^s_{,0}(I;L^2(\domain)) &\simeq L^2(\domain) \otimes H^s_{,0}(I) \quad \text{ for } s \in [0,1],
\end{align*}
where $\otimes$ signifies the Hilbertian tensor-product of separable 
Hilbert spaces, see e.g.\ \cite[Section~2.4]{ZankDissBuch2020} for details.
Last, we introduce the intersection spaces
\begin{align*}
H^{1,1/2}_{\Gamma_\mathrm{D};0,}(Q) 
=& 
\big(H^1_{\Gamma_\mathrm{D}}(\domain) \otimes L^2(I)\big)
\cap \big( L^2(\domain) \otimes H^{1/2}_{0,}(I) \big), \\
H^{1,1/2}_{\Gamma_\mathrm{D};,0}(Q) 
=& 
\big(H^1_{\Gamma_\mathrm{D}}(\domain) \otimes L^2(I)\big)
\cap \big(L^2(\domain) \otimes H^{1/2}_{,0}(I) \big),
\end{align*}
equipped with the sum norm, 
and the duality pairing 
\begin{equation}\label{eq:QDual}
\langle\cdot,\cdot\rangle_Q: \big[H^{1,1/2}_{\Gamma_\mathrm{D};,0}(Q)\big]' \times H^{1,1/2}_{\Gamma_\mathrm{D};,0}(Q) \to \IR
\end{equation}
as continuous extension of the $L^2(Q)$ inner product.

\subsection{Well-posedness and isomorphism results}
\label{sec:WellPos}
In this subsection, 
we state the space-time variational setting for the IBVP~\eqref{Einf:PDG}. 
First, we introduce the spatial operators $\mathcal L_x$, $\gamma_{x,0}$, $\gamma_{x,1}$, 
occurring in IBVP~\eqref{Einf:PDG}. 
We assume that the spatial differential operator $\mathcal L_x$ is 
linear, self-adjoint, and in divergence form, i.e.
\begin{equation} \label{DiffOpx}
  \mathcal L_x = -\nabla_x \cdot(A(\cdot) \nabla_x).
\end{equation}
Here, the diffusion coefficient $A\in L^\infty(\domain;\IR^{n\times n}_\mathrm{sym})$ is a symmetric and positive definite matrix function $A(x)$ of $x\in \domain$,
which does not depend on the temporal variable $t$. Further, we assume uniform positive definiteness of~$A$:
\be\label{eq:Acoerc}
\essinf_{x\in \domain} 
          \inf_{0 \ne \xi\in \IR^n} \frac{\xi^\intercal A(x) \xi}{\xi^\intercal\xi} > 0  
\;.
\ee
Next, $\gamma_{x,0}$
is the spatial Dirichlet trace operator, which is given by $\gamma_{x,0}(v) = v|_{\partial \domain \times I}$ for a continuous function $v$ on $\overline{Q}$. Moreover, $\gamma_{x,1}$ denotes the spatial conormal trace map, 
given for a sufficiently smooth function $v$ by $\gamma_{x,1}(v) = n_x \cdot (A(\cdot) \nabla_x v)|_{\partial \domain \times I}$ 
with the exterior unit normal vector $n_x \in L^\infty(\partial \domain;\IR^n)$ of $\domain$.

The space-time variational formulation of IBVP~\eqref{Einf:PDG} is to find $u \in H^{1,1/2}_{\Gamma_\mathrm{D};0,}(Q)$ such that
\begin{equation} \label{eq:IBVPVar}
  \forall w \in H^{1,1/2}_{\Gamma_\mathrm{D};,0}(Q)  : \quad b(u,w) = \langle f, w \rangle_Q,
\end{equation}
where $f\in \big[H^{1,1/2}_{\Gamma_\mathrm{D};,0}(Q)\big]'$ is given. Here, $b(\cdot, \cdot) \colon \, H^{1,1/2}_{\Gamma_\mathrm{D};0,}(Q) \times H^{1,1/2}_{\Gamma_\mathrm{D};,0}(Q) \to \IR$ is the continuous bilinear form given by
\begin{equation} \label{bilinearform}
    b(u,w) = \langle \partial_t u , w \rangle_Q  + \langle A \nabla_x u, \nabla_x w \rangle_{L^2(Q)^n}.
\end{equation}
Proceeding as in \cite[Theorem~3.2]{SteinbachZankETNA2020} or \cite[Theorem~3.4.19]{ZankDissBuch2020}, the space-time variational formulation~\eqref{eq:IBVPVar} has a unique solution. We summarize this result in the following theorem.
\begin{theorem}\label{thm:Biso}
Assume that $|\Gamma_\mathrm{D}|>0$ if $n=2,3$ or $\Gamma_\mathrm{D}\ne\emptyset$ if $n=1$, 
and that the spatial differential operator $\mathcal L_x$ in \eqref{DiffOpx} with the diffusion coefficient
$A\in L^\infty(\domain;\IR^{n\times n}_{\mathrm{sym}})$ 
satisfies~\eqref{eq:Acoerc}.

Then, 
the space-time variational formulation~\eqref{eq:IBVPVar} of IBVP~\eqref{Einf:PDG} is uniquely solvable, and
induces an isomorphism~ $\partial_t + \mathcal L_x  \in \cL_{\mathrm{iso}}
\big(H^{1,1/2}_{\Gamma_\mathrm{D};0,}(Q),\big[H^{1,1/2}_{\Gamma_\mathrm{D};,0}(Q)\big]'\big).$
\end{theorem}
\begin{remark}\label{rmk:HomBC}
  Inhomogeneous Neumann boundary conditions on $\Gamma_\mathrm{N}$ can be incorporated in a weak sense, resulting in an additional term on the right-hand side of the variational formulation~\eqref{eq:IBVPVar}.
  
  Inhomogeneous Dirichlet boundary conditions on $\Gamma_\mathrm{D}$ and initial conditions can be treated by using inverse trace theorems and by superposition since the IBVP~\eqref{Einf:PDG} is linear.
\end{remark}

\section{Discretization}
\label{sec:Discr}
We present finite-dimensional subspaces of the Bochnerian 
function spaces on $Q$ which shall be used in the ensuing 
space-time Galerkin discretizations at the end of this section.
As all subspaces are algebraic tensor-products of subspaces 
on $I=(0,T)$ and on $\domain$, we introduce the factor spaces 
separately. Without loss of generality, we assume $T=1$, 
i.e. $I=(0,1)$.

\subsection{Tensor-product trial spaces on $Q = \domain \times I$}
\label{Sec:StwLineare}
In this subsection, 
we introduce the temporal and spatial finite element spaces 
as well as their tensor-product.

\subsubsection{Temporal trial space}
\label{sec:TempFESpc}
For a refinement level $j \in \IN_0$ and 
given $N_j^t \in \IN$, 
consider in $I=(0,1)$ the 
temporal uniform mesh
\[
  \mathcal T_j^t = \big\{ \tau_{j,\ell}^t \subset I: \; 
  \tau_{j,\ell}^t =(t_{j,\ell-1},t_{j,\ell}) \text{ for } \ell=1,\dots,N_j^t\big\}
\]
given by $t_{j,\ell} = \ell/N_j^t$ for $\ell=0,\dots,N_j^t$ with uniform 
mesh size $h_j^t = 1/N_j^t$. 
With the temporal mesh $\mathcal T_j^t$, 
we associate the trial space
\begin{equation*}
  S_j^t = \{ v_j^t\in C(\overline{I}): \, \forall \tau \in \mathcal{T}_t^j:
  \, v_j^t|_{\overline{\tau}} \in \mathbb P^{d-1}(\overline{\tau})\text{ and } v_j^t(0) = 0\}
\end{equation*}
of globally continuous, piecewise polynomial functions, 
fulfilling the homogeneous initial condition. 
Here, $\mathbb P^{d-1}(A)$ denotes the space of polynomials of order
$d$ on a set $A \subset \IR^m$, $m \in \IN$. 
For $d=2$,
we obtain the usual hat functions $\phi_{j,k}^t:\overline{I}\to\IR$, 
$k=1,\dots,N_j^t$, fulfilling $\phi_{j,k}^t(t_{j,\ell})=\delta_{k,\ell}$ 
for $\ell=0,\dots,N_j^t$ with the Kronecker delta $\delta_{k,\ell}$.
For $d>2$, the present setting includes finite elements of higher 
order and B-splines. Note that in case of B-splines appropriate 
wavelet bases used later for matrix compression
are available.

\subsubsection{Spatial trial space}
\label{sec:SpatFESpc}
To define the spatial trial space in $\domain$, 
we assume $\domain \subset \IR^n$, $n \in \{1,2,3\}$, 
to be a given bounded Lipschitz domain, which is an interval 
$\domain=(0,L)$ with $L>0$ for $n=1$, polygonal for $n=2$, 
or polyhedral for $n=3$. For a refinement level $j\in \IN_0$, 
we decompose $\domain$ into an admissible decomposition
\[
  \mathcal{T}_j^x = \{ \tau_{j,i}^x\subset \IR^n: i=1,\dots,n_j^x\},
\]
where $n_j^x$ is the number of spatial elements $\tau_{j,i}^x$. 
We denote by $h_{j,i}^x$, $i=1,\dots,n_j^x$, the local mesh 
size and $h_j^x = \max_i h_{j,i}^x$ is the maximal mesh size. 
To a shape-regular sequence $(\mathcal T_j^x)_{j\in \IN_0}$, 
we relate the spatial finite element spaces 
\[
  S_j^x = \big\{ w_j^x\in C(\overline{\domain}):
  \forall \tau\in\mathcal T_j^x : w_j^x|_{\overline{\tau}}
  \in \mathbb P^{p_x}(\overline{\tau}) \text{ and } w_j^x|_{\Gamma_\mathrm{D}} = 0 \big\}
\]
of functions, which are globally continuous, 
piecewise polynomial functions of degree $p_x \ge 1$ and 
fulfill the homogeneous Dirichlet condition on $\Gamma_\mathrm{D}$. 
We denote by 
$\phi_{j,k}^x:\overline{\domain} \to \IR$, $k=1,\dots, N_j^x$, 
basis functions of $S_j^x$, i.e. we have that
$S_j^x = \spn\{\phi_{j,k}^x : k=1,\dots,N_j^x\}$.

\subsubsection{Space-time tensor-product trial space}
\label{sec:SpaceTime}
The space-time trial space is obtained via
the tensor-product of the trial spaces in 
space and time. 
To this end, fixing a level $j \in \IN_0$, 
we define the space-time partition of $Q$
\[
  \mathcal T_j^{x,t} = \big\{ \tau^x \times \tau^t \subset \IR^{n+1}:
  \tau^x\in\mathcal T_j^x, \, \tau^t\in\mathcal T_j^t \big\}
\]
with space-time mesh size $h_j = \max\{h_j^x,h_j^t\}$.
The respective space-time trial space is given by
\begin{equation*}
  S_j^{x,t} = S_j^x\otimes S_j^t \subset H^{1,1/2}_{\Gamma_\mathrm{D};0,}(Q),
\end{equation*}
fulfilling the homogeneous Dirichlet and initial conditions. 
A basis is via dyads of the basis functions in space and time
(being here equivalent to pointwise products), 
that is
\[
  S_j^{x,t} = \spn\big\{\phi_{j,k}^x \otimes \phi_{j,\ell}^t:
  			k=1,\ldots,N_j^x,\,\ell=1,\ldots,N_j^t\big\}.
\]

\subsection{Conforming space-time Galerkin discretization}
\label{Sec:FEMxt}
We shall next introduce the conforming 
space-time Galerkin discretization of the parabolic evolution equation~\eqref{Einf:PDG}
by employing the modified Hilbert transform $\HT$,
which leads to the $H^{1/2}_{0,}(I; L^2(\domain))$-ellipticity of the bilinear form $b(\cdot,\HT\cdot)$
in \eqref{bilinearform}. More precisely, tensorizing \eqref{HT:Elliptisch} and 
\eqref{HTPositivsemi} with $L^2(\domain)$ we have that (cp.\ \eqref{eq:QDual})
\begin{align}
  \| v \|_{H^{1/2}_{0,}(I; L^2(\domain))}^2 &= \langle \partial_t v, \HT v \rangle_Q \nonumber \\
  &\leq  \langle \partial_t v, \HT v \rangle_Q + \langle A \nabla_x v, \nabla_x \HT v \rangle_{L^2(Q)^n} 
     = b(v,\HT v) 
\label{b_H12elliptisch}
\end{align}
for all $v \in H^{1,1/2}_{\Gamma_\mathrm{D};0,}(Q)$.

Then, for any refinement level $j \in \IN_0$, the 
full tensor-product space-time Galerkin
finite element method to find 
$u_j\in S_j^{x,t} \subset H^{1,1/2}_{\Gamma_\mathrm{D};0,}(Q)$ 
such that
\begin{equation} \label{FEMxt:WaermeFEM}
  \forall v_j\in S_j^{x,t}: \quad 
  b(u_j, \HT v_j)
  = \spf{f}{\HT v_j}_{Q}
\end{equation}
is well-defined, with $b(\cdot,\cdot)$ as in \eqref{bilinearform}.
The tensor-product space-time Galerkin projection $u_j\in S_j^{x,t}$ 
in \eqref{FEMxt:WaermeFEM} is unconditionally stable with 
the space-time stability estimate
\begin{equation*}
  \norm{u_j}_{ H^{1/2}_{0,}(I;L^2(\domain))} 
  \leq \norm{f}_{[H^{1/2}_{,0}(I;L^2(\domain))]'},
\end{equation*}
provided that $f \in [H^{1/2}_{,0}(I;L^2(\domain))]'$, 
see \cite[Theorem~3.4.20]{ZankDissBuch2020}.
%
\begin{theorem}\label{thm:convergence}
Let $j \in \IN_0$ be a given refinement level. 
In addition, 
let $u \in H^{1,1/2}_{\Gamma_\mathrm{D};0,}(Q)$ be the 
unique solution of the space-time variational formulation~\eqref{eq:IBVPVar}.
Suppose the regularity assumption $\mathcal L_x u \in L^2(Q)$ holds.
Let $u_j \in S_j^{x,t}$ be the space-time approximation given by~\eqref{FEMxt:WaermeFEM}.

Then, the space-time error estimate
  \begin{equation} \label{thm:Fehlerabsch}
    \| u - u_j \|_{H^{1/2}_{0,}(I; L^2(\domain))} 
        \leq  2\| u - P_j u \|_{H^{1/2}_{0,}(I; L^2(\domain))} 
         + \big\| (\mathrm{Id}_{xt} - P_j^t \otimes \mathrm{Id}_x) \mathcal L_x u \big\|_{[H^{1/2}_{,0}(I; L^2(\domain))]'}
  \end{equation}
holds true with $P_j u = (P_j^t \otimes P_j^x) u \in S_j^{x,t}$,
where $P_j^t \colon \, L^2(I) \to S_j^t$ 
is any temporal projection
and 
$P_j^x$ denotes the spatial Ritz projection 
$P_j^x \colon \, H^1_{\Gamma_\mathrm{D}}(\domain)\to S_j^x$.
It is defined as follows: given $z \in H^1_{\Gamma_\mathrm{D}}(\domain)$,
find $P_j^x z \in S_j^x$ such that
  \begin{equation}
\label{eq:Ritz}
    \forall w_j \in S_j^x: \quad \langle A \nabla_x P_j^x z, \nabla_x w_j \rangle_{L^2(\Omega)^n} 
    	= \langle A \nabla_x z, \nabla_x w_j \rangle_{L^2(\Omega)^n}.
  \end{equation}
\end{theorem}
\begin{proof}
The proof is a generalization of 
the proof of 
\cite[Theorem~3.4]{SteinbachZankETNA2020} or 
\cite[Theorem~3.4.24]{ZankDissBuch2020}. 
The triangle inequality yields
\begin{equation*}
  \| u - u_j \|_{H^{1/2}_{0,}(I; L^2(\domain))} \leq \| u - P_j u \|_{H^{1/2}_{0,}(I; L^2(\domain))} 
  + \| u_j - P_j u \|_{H^{1/2}_{0,}(I; L^2(\domain))}.
\end{equation*}
The last term is estimated as follows. 

Using the $H^{1/2}_{0,}(I; L^2(\domain))$-ellipticity 
of the bilinear form $b(\cdot,\cdot)$ in~\eqref{HT:Elliptisch} and \eqref{HTPositivsemi}
with the Galerkin orthogonality corresponding 
to \eqref{FEMxt:WaermeFEM} give
\begin{align*}
    \| u_j - P_j u \|_{H^{1/2}_{0,}(I; L^2(\domain))}^2 
    &\leq  b\big(u_j - P_j u, \HT (u_j - P_j u )\big) \\
    &= b\big( u - P_j u, \HT (u_j - P_j u)\big) \\
    &= \langle \partial_t (u- P_j u), \HT (u_j - P_j u) \rangle_Q \\
    &+ \langle A \nabla_x u, \nabla_x \HT (u_j - P_j u)\rangle_{L^2(Q)^n} \\
    &- \langle A \nabla_x (P_j^t \otimes P_j^x) u, \nabla_x \HT (u_j - P_j u) \rangle_{L^2(Q)^n}
   \;.
\end{align*}
From property~\eqref{HT:Beschr} of $\mathcal H_T$ and 
the Ritz projection property \eqref{eq:Ritz} of $P_j^x$, 
it follows from the preceding bound with integration 
by parts in $\domain$ that
\begin{align*}
    \| u_j - P_j u \|_{H^{1/2}_{0,}(I; L^2(\domain))}^2 
&\leq \| u - P_j u \|_{H^{1/2}_{0,}(I; L^2(\domain))} \| u_j - P_j u \|_{H^{1/2}_{0,}(I; L^2(\domain))} \\
    &+ \langle A \nabla_x u, \nabla_x \HT (u_j - P_j u)\rangle_{L^2(Q)^n} \\
    &- \langle A \nabla_x(P_j^t\otimes \mathrm{Id}_x) u, \nabla_x \HT (u_j - P_j u) \rangle_{L^2(Q)^n} \\
    &= \| u - P_j u \|_{H^{1/2}_{0,}(I; L^2(\domain))} \| u_j - P_j u \|_{H^{1/2}_{0,}(I; L^2(\domain))} \\
    &- \langle \mathcal L_x u, \HT (u_j - P_j u)\rangle_{L^2(Q)} \\
    &+ \langle \big(P_j^t \otimes \mathrm{Id}_x \big) \mathcal L_xu, \HT (u_j - P_j u) \rangle_{L^2(Q)}.
\end{align*}
Duality and the isometry property of $\HT$
finally yield
\begin{multline*}
   \| u_j - P_j u \|_{H^{1/2}_{0,}(I; L^2(\domain))}^2 
\leq 
\| u - P_j u \|_{H^{1/2}_{0,}(I; L^2(\domain))} \| u_j - P_j u \|_{H^{1/2}_{0,}(I; L^2(\domain))} 
\\
 +\big\| \big(\mathrm{Id}_{xt} - P^t_j \otimes \mathrm{Id}_x \big) \mathcal L_x u \big\|_{[H^{1/2}_{,0}(I; L^2(\domain))]'} 
 \| u_j - P_j u \|_{H^{1/2}_{0,}(I; L^2(\domain))},
\end{multline*}
which implies the assertion.
\end{proof}

To prove convergence rates given in the forthcoming 
Corollary~\ref{cor:FullTPRate}, we define the $H^{1/2}_{0,}(I)$-projection 
$P_j^{1/2,t} \colon \, H^{1/2}_{0,}(I) \to S_j^t$ by
\begin{equation} \label{H12-Projektion}
  \forall z_j \in S_j^t : \quad  \langle \partial_t P_j^{1/2,t} v, 
  	\HT z_j \rangle_{L^2(I)} = \langle \partial_t v, \HT z_j \rangle_{I}
\end{equation}
for a given function $v \in H^{1/2}_{0,}(I)$, fulfilling the error estimate
\[
  \| v - P_j^{1/2,t} v \|_{H^{1/2}_{0,}(I)} \leq c (h_j^t)^{d-1/2} \|v\|_{H^d(I)}
\]
provided that $v \in H^{1/2}_{0,}(I) \cap H^d(I)$ with a constant 
$c>0$. Further, setting the function $w = \int_0^\cdot \HT(v - P_j^{1/2,t} v)(s) \mathrm ds
\in H^1_{0,}(I)$, i.e.\ an Aubin--Nitsche-type argument, yields
\begin{align*}
    \| v - P_j^{1/2,t} v \|_{L^2(I)}^2 &= \langle \HT (v - P_j^{1/2,t} v) , \HT (v - P_j^{1/2,t} v) \rangle_{L^2(I)} \\
    &= \langle \partial_t w , \HT (v - P_j^{1/2,t} v) \rangle_{L^2(I)} \\
    &= \langle \partial_t (w - P_j^{1/2,t} w), \HT (v - P_j^{1/2,t} v) \rangle_{L^2(I)} \\
    &\leq \| w - P_j^{1/2,t} w \|_{H^{1/2}_{0,}(I)}  \| v - P_j^{1/2,t} v \|_{H^{1/2}_{0,}(I)} \\
    &\leq c (h_j^t)^{1/2} \| \partial_t w \|_{L^2(I)} \| v - P_j^{1/2,t} v \|_{H^{1/2}_{0,}(I)}  \\
    &= c (h_j^t)^{1/2} \| v - P_j^{1/2,t} v \|_{L^2(I)} \| v - P_j^{1/2,t} v \|_{H^{1/2}_{0,}(I)},
\end{align*}
where the isometry property of $\HT$, the fundamental theorem 
of calculus, the boundedness~\eqref{HT:Beschr} and an 
$H^{1/2}_{0,}(I)$-error estimate of the $H^{1/2}_{0,}(I)$-projection 
are used. Thus, we arrive at the $L^2(I)$-error estimate
\begin{equation} \label{H12Projektion_L2fehler}
  \| v - P_j^{1/2,t} v \|_{L^2(I)} \leq c (h_j^t)^d \|v\|_{H^d(I)}
\end{equation}
whenever $v \in H^{1/2}_{0,}(I) \cap H^d(I)$.

\begin{corollary}
\label{cor:FullTPRate}
Let the assumptions of Theorem~\ref{thm:convergence} be satisfied. 
Consider temporal biorthogonal spline-wavelets $S_j^t$ 
of order $d\geq 2$ in $I$ with uniform mesh size $h_j^t$
and 
Lagrangian finite elements $S_j^x$ of degree $p_x \geq 1$ 
on a regular, simplicial partition of $\domain$ 
with maximal mesh size $h_j^x$.

Assume also 
that the exact solution $u$
of the variational IBVP \eqref{eq:IBVPVar} 
admits the regularity 
$u\in H^{1/2}_{0,}(I; H^{p_x+1/2}(\domain)) \cap H^{d}(I;L^2(\domain))$ 
and 
$\mathcal L_x u \in H^{d-1}(I;L^2(\domain))$. 
Assume finally that the spatial
Ritz projection $P_j^x \colon \, H^1_{\Gamma_\mathrm{D}}(\domain)\to S_j^x$ 
satisfies
\begin{equation} \label{RitzFehlerL2}
\| u - (\mathrm{Id}_t\otimes P_j^x) u \|_{H^{1/2}_{0,}(I; L^2(\domain))} 
\leq 
c (h_j^x)^{p_x+1/2} \| u \|_{H^{1/2}_{0,}(I; H^{p_x+1/2}(\domain))}
\end{equation}
with a constant $c>0$.

Then, the space-time error estimate
$$
 \| u - u_j \|_{H^{1/2}_{0,}(I; L^2(\domain))}
  \le 
C [ (h_j^t)^{d-1/2} + (h_j^x)^{p_x+1/2} ]
$$
holds true with a constant $C>0$.
\end{corollary}

\begin{proof}
Let the temporal projection $P_j^t$ be the $L^2(I)$-projection on $S_j^t$. 
Using the triangle inequality for the first part on the right-hand side 
of the estimate~\eqref{thm:Fehlerabsch} yields
\begin{align*}
    \| u - P_j u \|_{H^{1/2}_{0,}(I; L^2(\domain))} 
    	\leq& \| (\mathrm{Id}_{xt} - \mathrm{Id}_t \otimes P_j^x) u \|_{H^{1/2}_{0,}(I; L^2(\domain))} 
    \\
    &+ \| (\mathrm{Id}_{xt} - P_j^t \otimes \mathrm{Id}_x) u \|_{H^{1/2}_{0,}(I; L^2(\domain))} 
    \\
    &+ \| (\mathrm{Id}_{xt} - P_j^t \otimes \mathrm{Id}_x)  (\mathrm{Id}_{xt} - \mathrm{Id}_t \otimes P_j^x)  u \|_{H^{1/2}_{0,}(I; L^2(\domain))}.
\end{align*}
For $\| (\mathrm{Id}_{xt} - \mathrm{Id}_t \otimes P_j^x) u \|_{H^{1/2}_{0,}(I; L^2(\domain))}$, 
we apply estimate~\eqref{RitzFehlerL2}. For $\| (\mathrm{Id}_{xt} 
- P_j^t \otimes \mathrm{Id}_x) u \|_{H^{1/2}_{0,}(I; L^2(\domain))}$, the 
triangle inequality upon inserting the temporal $H^{1/2}_{0,}(I)$-projection 
$P_j^{1/2,t}$ given in \eqref{H12-Projektion}
and then using the temporal inverse inequality
\[
     \forall z_j \in S_j^t: \| z_j \|_{H^{1/2}_{0,}(I)} \leq c_{\mathrm{I}} (h_j^t)^{-1/2} \| z_j \|_{L^2(I)}
\]
give the error bound
\begin{align*}
  \| (\mathrm{Id}_{xt} - P_j^t \otimes \mathrm{Id}_x) &u \|_{H^{1/2}_{0,}(I; L^2(\domain))} \leq \| (\mathrm{Id}_{xt} - P_j^{1/2,t} \otimes \mathrm{Id}_x) u \|_{H^{1/2}_{0,}(I; L^2(\domain))} \\
  &+ \| (P_j^{1/2,t} \otimes \mathrm{Id}_x - P_j^t \otimes \mathrm{Id}_x) u \|_{H^{1/2}_{0,}(I; L^2(\domain))} \\
  \leq& c_1 (h_j^t)^{d-1/2} + c_{\mathrm{I}} (h_j^t)^{-1/2} \| (P_j^{1/2,t} \otimes \mathrm{Id}_x - \mathrm{Id}_{xt}) u \|_{L^2(Q)} \\
  &+ c_{\mathrm{I}} (h_j^t)^{-1/2} \| (\mathrm{Id}_{xt} - P_j^t \otimes \mathrm{Id}_x) u \|_{L^2(Q)} \\  
  \leq& c_2 (h_j^t)^{d-1/2}
\end{align*}
with constants $c_1, c_2>0$, when $L^2(I)$-error estimates for 
the $L^2(I)$- and $H^{1/2}_{0,}(I)$-projection are applied, where 
the latter follows from the Aubin--Nitsche argument~\eqref{H12Projektion_L2fehler}
for the $H^{1/2}_{0,}(I)$-projection. For the term
$\| (\mathrm{Id}_{xt} - P_j^t \otimes \mathrm{Id}_x)  
(\mathrm{Id}_{xt} - \mathrm{Id}_t \otimes P_j^x)  u \|_{H^{1/2}_{0,}(I; L^2(\domain))}$, 
we employ the triangle inequality to get
\begin{multline*}
 \| (\mathrm{Id}_{xt} - P_j^t \otimes \mathrm{Id}_x)  
 (\mathrm{Id}_{xt} - \mathrm{Id}_t \otimes P_j^x)  u \|_{H^{1/2}_{0,}(I; L^2(\domain))} 
 \leq  \| (\mathrm{Id}_{xt} - \mathrm{Id}_t \otimes P_j^x)  u \|_{H^{1/2}_{0,}(I; L^2(\domain))} \\
 + \| (P_j^t \otimes \mathrm{Id}_x)  (\mathrm{Id}_{xt} - \mathrm{Id}_t \otimes P_j^x)  u \|_{H^{1/2}_{0,}(I; L^2(\domain))}.
\end{multline*}
For the first term, we use error estimate~\eqref{RitzFehlerL2}. 
For the second term, we apply the stability of the $L^2(I)$-projection 
in $H^{1/2}_{0,}(I)$ and again, error estimate~\eqref{RitzFehlerL2}. 
With this, we derive the error bound $\| (\mathrm{Id}_{xt} 
- P_j^t \otimes \mathrm{Id}_x)  (\mathrm{Id}_{xt} 
- \mathrm{Id}_t \otimes P_j^x)  u \|_{H^{1/2}_{0,}(I; L^2(\domain))} 
\leq c (h_j^x)^{p_x+1/2} \| u \|_{H^{1/2}_{0,}(I; H^{p_x+1/2}(\domain))}$ 
with a constant $c>0$. Thus, we conclude that
\begin{multline*}
\| u - P_j u \|_{H^{1/2}_{0,}(I; L^2(\domain))} 
\leq 
c (h_j^x)^{p_x+1/2} \| u \|_{H^{1/2}_{0,}(I; H^{p_x+1/2}(\domain))} 
\\
+ c (h_j^t)^{d-1/2} \| u \|_{H^{d}(I;L^2(\domain))}
\end{multline*}
with a constant $c>0$.

For the second part on the right-hand side of the 
estimate~\eqref{thm:Fehlerabsch}, 
we use the standard error estimate of the 
temporal $L^2(I)$-projection with respect to $\| \cdot \|_{[H^{1/2}_{,0}(I)]'}$ 
to conclude that  
\[
\big\| (\mathrm{Id}_{xt} - P_j^t \otimes \mathrm{Id}_x) \mathcal L_x u \big\|_{[H^{1/2}_{,0}(I; L^2(\domain))]'} 
\leq 
c (h_j^t)^{d-1/2} \| \mathcal L_x u \|_{H^{d-1}(I;L^2(\domain))}
\]
with a constant $c>0$.
\end{proof}

\begin{remark}
The error estimate~\eqref{RitzFehlerL2} can be proven by 
an Aubin--Nitsche argument and function space interpolation, 
provided that the unique solution $w_g \in H^1_{\Gamma_\mathrm{D}}
(\domain)$ of the spatial variational formulation
\[
  \forall z \in H^1_{\Gamma_\mathrm{D}}(\domain): \quad 
  \langle A \nabla_x z, \nabla_x w_g \rangle_{L^2(\domain)^n} = \langle g, z \rangle_{L^2(\domain)}
\]
for a given $g \in L^2(\domain)$ satisfies the regularity 
assumption $w_g \in H^2(\domain)$ with estimate $\| w_g \|_{H^2(\domain)} 
\leq C \| g \|_{L^2(\domain)}$ with a constant $C> 0$. 
\end{remark}

The discrete variational formulation~\eqref{FEMxt:WaermeFEM} 
is equivalent to the global linear system
\begin{equation*}
  \big({\bf A}_j^t \otimes {\bf M}_j^x + {\bf M}_j^t\otimes {\bf A}_j^x\big) {\bf u}_j = {\bf f}_j
\end{equation*}
with the Kronecker matrix product $\otimes$.
Here, 
the temporal matrices given by
\begin{equation*}
    [{\bf A}_j^t]_{k,k'} = \spf{\partial_t \phi_{j,k'}^t}{\HT \phi_{j,k}^t}_{L^2(I)}, 
    \quad [{\bf M}_j^t]_{k,k'} = \spf{\phi_{j,k'}^t}{\HT \phi_{j,k}^t}_{L^2(I)}
\end{equation*}
for $k,k'=1,\dots,N_j^t$ 
and 
the spatial matrices are
\begin{equation*}
  [{\bf M}_j^x]_{k,k'} = \spf{\phi_{j,k'}^x}{\phi_{j,k}^x}_{L^2(\domain)}, 
  \quad [{\bf A}_j^x]_{k,k'} = \spf{\nabla_x \phi_{j,k'}^x}{\nabla_x \phi_{j,k}^x}_{L^2(\domain)^n}
\end{equation*}
for $k,k'=1,\dots,N_j^x$, where ${\bf f}_j$ is the related 
right-hand side.

The matrices ${\bf A}_j^t$, ${\bf M}_j^t$,
${\bf A}_j^x$, ${\bf M}_j^x$ are positive definite, see the 
properties of $\HT$ for the temporal matrices. 
Moreover, 
the matrices ${\bf A}_j^t$, ${\bf A}_j^x$, ${\bf M}_j^x$ are 
symmetric, whereas ${\bf M}_j^t$ is nonsymmetric.

\section{Wavelets and multiresolution analysis}
	\label{sec:wavelets}
Instead of using the single-scale basis in $S_j^t$ for
the temporal discretization, we shall switch to a wavelet 
basis in order to apply wavelet matrix compression for 
the data sparse representation of the dense matrices 
${\bf A}_j^t$ and ${\bf M}_j^t$. To this end, we assume 
without loss of generality that $T=1$ and consider the 
interval $I =(0,1)$. Moreover, for the sake of simplicity in 
representation, we omit the suffix $t$ in this section. 

A \emph{multiresolution analysis} (MRA for short) 
consists of a nested family of finite dimensional 
approximation spaces
\begin{equation}	
\label{eq:hierarchy}
\{0\} = S_{-1} \subset S_0 \subset S_1 \subset \cdots 
\subset S_j\subset \cdots \subset L^2(I),
\end{equation}
such that
\[ \overline{\bigcup_{j\ge0} S_j} = L^2(I)\quad\text{and}\quad
\dim S_j\sim 2^{j}.
\] We will
refer to \(j\) as the \emph{level} of \(S_j\) in the
multiresolution analysis.
Each space $S_j$ is endowed with a \emph{single-scale basis}
\[
\Phi_j = \{\phi_{j,k}:k\in\Delta_j\},
\]
i.e.\ $S_j=\operatorname{span}\Phi_j$, where $\Delta_j$ denotes 
a suitable index set with cardinality $|\Delta_j|=\dim S_j$. For 
convenience, we write  in the sequel bases in the form of row vectors, 
such that, for ${\bf v} = [v_k]_{k\in\Delta_j}\in \ell^2(\Delta_j)$, the 
corresponding function can simply be written as a dot product according to
\[
v_j = \Phi_j{\bf v}=\sum_{k\in\Delta_j} v_k\phi_{j,k}. 
\]
In addition, we shall assume that the single-scale 
bases $\Phi_j$ are \emph{uniformly stable}, this means that 
\[
\|{\bf v}\|_{\ell^2(\Delta_j)}
\sim \|\Phi_j{\bf v}\|_{L^2(I)}\quad\text{for all }
{\bf v}\in \ell^2(\Delta_j)
\]
uniformly in $j$, and that they satisfy the locality condition 
\[
\operatorname{diam}(\operatorname{supp}\phi_{j,k})\sim 2^{-j}.
\]

Additional properties of the spaces $S_j$ are required
for using them as trial spaces in a Galerkin scheme.
The temporal approximation spaces shall have the 
\emph{regularity} 
\[
\gamma=\sup\{s\in\mathbb{R}: S_j\subset H^s(I)\}
\]
and the 
\emph{approximation order} $d\in\mathbb{N}$, 
that is
\[
       d = \sup\Big\{s\in\mathbb{R}: \inf_{v_j\in S_j}\|v-v_j\|_{L^2(I)} 
       		\lesssim 2^{-js}\|v\|_{H^s(I)}\Big\}.
\]

Rather than using the multiresolution analysis corresponding 
to the hierarchy in \eqref{eq:hierarchy}, the pivotal idea 
of wavelets is to keep track of the increment of information 
between two consecutive levels $j-1$ and $j$. 
Since we have
$S_{j-1}\subset S_j$, we may decompose 
\[
S_j = S_{j-1}\oplus W_j,\ \text{i.e.}\ S_{j-1}\cup W_j = S_{j}\ \text{and}\  S_{j-1}\cap W_j = \{0\},
\]
with an appropriate \emph{detail space} $W_j$. Of practical interest 
is the particular choice of the basis of the detail space  $W_j$ in $S_j$. 
This basis will be denoted by
\[
  \Psi_j = \{\psi_{j,k}: k \in \nabla_j\}
\]
with the index set $\nabla_j= \Delta_j\setminus \Delta_{j-1}$.
In particular, we shall assume that the collections $\Phi_{j-1}\cup\Psi_j$
form uniformly stable bases of $S_j$, as well. If $\Psi = 
\bigcup_{j\ge 0}\Psi_j$, where $\Psi_0= \Phi_0$, is even a 
Riesz-basis of $L^2(I)$, then it is called a \emph{wavelet 
basis}. We require the functions $\psi_{j,k}$ to be 
localized with respect to the corresponding level $j$, i.e.\ 
\[
\operatorname{diam}(\operatorname{supp}\psi_{j,k}) \sim 2^{-j}, 
\]
and we normalize them such that 
\[
\|\psi_{j,k}\|_{L^2(I)}\sim 1.
\]

At first glance, it would be very convenient to deal with a single 
orthonormal system of wavelets. However, it has been shown in 
\cite{DHS1,DPS4,SchneiderBuch1998} that orthogonal wavelets are 
not optimal for the efficient approximation nonlocal operator equations. 
For this reason, we rather use \emph{biorthogonal wavelet bases}.
In this case, we also have a dual, multiresolution
analysis, i.e.\ dual single-scale bases and wavelets
\[
\widetilde{\Phi}_j = \{\widetilde{\phi}_{j,k}:k\in\Delta_j\},\quad
\widetilde{\Psi}_j = \{\widetilde{\psi}_{j,k}:k\in\nabla_j\},
\] 
respectively.
They are biorthogonal to the primal ones: there hold
the orthogonality conditions 
\[
\langle\Phi_j,\widetilde{\Phi}_j\rangle_{L^2(I)} = {\bf I},\quad
\langle\Psi_j,\widetilde{\Psi}_j\rangle_{L^2(I)} = {\bf I}. 
\]
The corresponding spaces $\widetilde{S}_j= \operatorname{span}\widetilde{\Phi}_j$
and $\widetilde{W}_j=\operatorname{span}\widetilde{\Psi}_j$ satisfy
\begin{equation}	\label{eq:space-coupling}
  S_{j-1}\perp \widetilde{W}_j, \quad \widetilde{S}_{j-1}\perp W_j.
\end{equation}
Moreover, the dual spaces are supposed to exhibit some approximation order
$\widetilde{d}\in\mathbb{N}$ and regularity $\widetilde{\gamma}>0$.

Denoting in complete analogy to the primal basis 
$\widetilde{\Psi} = \bigcup_{j\ge 0}\widetilde{\Psi}_j$, where $\widetilde{\Psi}_0 
=\widetilde{\Phi}_0$, 
then every $v \in L^2(I)$ has unique representations 
\[
v = \sum_{j\ge 0}\sum_{k\in\nabla_j} \langle v,\widetilde\psi_{j,k}\rangle_{L^2(I)}\psi_{j,k}
= \sum_{j\ge 0}\sum_{k\in\nabla_j} \langle v,\psi_{j,k}\rangle_{L^2(I)}\widetilde\psi_{j,k}
\]
such that
\[
 \|v\|_{L^2(I)}^2 \sim \sum_{j\ge 0}\sum_{k\in\nabla_j}
	\big\|\langle v,\widetilde\psi_{j,k}\rangle_{L^2(I)}\big\|_{\ell^2(\nabla_j)}^2
 \sim \sum_{j\ge 0}\sum_{k\in\nabla_j}
	\big\|\langle v,\psi_{j,k}\rangle_{L^2(I)}\big\|_{\ell^2(\nabla_j)}^2.
\]
In particular, relation \eqref{eq:space-coupling} implies that the 
wavelets exhibit {\em vanishing moments} of order $\widetilde{d}$, i.e.\
\begin{equation*}
  \big|\langle v,\psi_{j,k}\rangle_{L^2(I)}\big| \lesssim 2^{-j(1/2+\widetilde{d})}
	|v|_{W^{\widetilde{d},\infty}(\operatorname{supp}\psi_{j,k})}
\end{equation*}
for all $v \in W^{\widetilde{d},\infty}(I)$. Herein, 
the quantity $|v|_{W^{\widetilde{d},\infty}(I)}
= \|\partial^{\widetilde{d}} v\|_{L^\infty(I)}$
is the semi-norm in $W^{\widetilde{d},\infty}(I)$.
We refer to \cite{DA1} for further details.

For the compression of the mass- and stiffness matrices of 
the operator $\cH_T$ presently considered, the $\Phi_j,\widetilde\Phi_j$ 
are generated by constructing first dual pairs of single-scale bases 
on the interval $[0,1]$, using B-splines for the primal bases and 
the dual components from \cite{CDF} adapted to the interval 
\cite{DKU}. Then, we follow the construction of \cite{DS} to incorporate
the homogeneous boundary condition in $t=0$. 
In case of continuous, piecewise linear spline-wavelets, i.e. $d=2$,
we obtain the wavelets shown in 
Figure~\ref{fig:(d,td)=(2,2)}, 
providing $\widetilde{d} = 2$ vanishing moments. 
The 
detailed refinement relation of these wavelets and also of 
piecewise linear spline-wavelets with $\widetilde{d} = 4$ 
vanishing moments are found in the appendix.

\begin{figure}
\begin{center}
\includegraphics[width=0.6\textwidth]{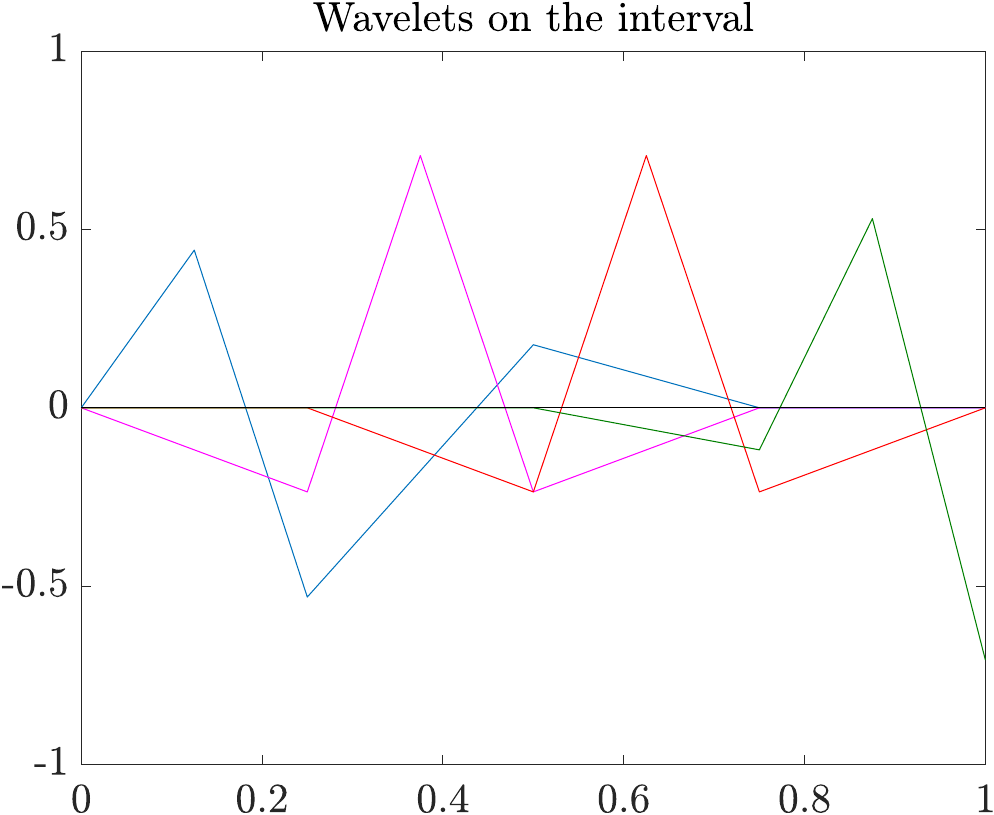}
\caption{\label{fig:(d,td)=(2,2)}
Continuous, piecewise linear wavelets on level $j=2$ 
in $I=(0,1)$
with two vanishing moments and a homogeneous boundary 
condition in $t=0$. Blue and green: boundary wavelets, 
accounting for homogeneous Dirichlet initial condition at $t=0$.
Orange and magenta: interior wavelets. }
\end{center}
\end{figure}

\section{Wavelet matrix compression}
	\label{sec:compression}
The basic ingredients in the analysis of the 
wavelet compression are decay
estimates for the matrix entries.
They are based on Calder\'{o}n--Zygmund-like
estimates of the kernel function in \eqref{HT:L2Kern} 
which we shall provide first.
They enable to apply results from \cite{DHS1} to define 
optimally compressed versions 
of the matrices ${\bf A}_j^t$ and ${\bf M}_j^t$.

\subsection{Calder\'{o}n--Zygmund estimates}
\label{sec:CZ}
For the analysis of wavelet Galerkin discretizations,
the splitting
\begin{equation} \label{eq:HTpm}
  \HT = \HT^+ + \HT^-,
\end{equation}
motivated by the kernel function in \eqref{HT:L2Kern}, will be important. 
In \eqref{eq:HTpm}, the operators $\HT^+$ and $\HT^-$ are given by
\[
  (\HT^\pm v)(t) = {\mathrm  {v.p.}} \int_0^1  v(s) \, K^\pm(s,t) \, 
  \mathrm ds, \quad t \in I,\ v \in L^2(I) \;.
\]
Here, ${\mathrm  {v.p.}}$ indicates the Cauchy principal value.
The kernel functions $K^\pm$ comprising $K$ in \eqref{HT:L2Kern} 
are given as
\begin{equation*}
  K^+(s,t) = \frac{1}{2 \sin \frac{\pi(s+t)}{2}  } 
  \quad \text{and} \quad  
  K^-(s,t) = \frac{1}{2 \sin \frac{\pi\abs{s-t}}{2}  }\;,
  \quad (s,t)\in I^2\backslash \{ s = t \}.
\end{equation*}
The property $\ln(ab) = \ln(a)+\ln(b)$ of the natural logarithm
implies in \eqref{eq:K-1} a decomposition
$K_{-1} = K_{-1}^+ + K_{-1}^-$ which corresponds to \eqref{eq:HTpm}.

The following properties of the kernels $K^\pm$ are the 
basis for the ensuing wavelet compression results.
\begin{proposition}\label{prop:CZO}
There exist constants $C,d>0$ such that
for every $(s,t) \in I^2\backslash \{ s=t \}$, and
for every $p,p'\in \IN$ hold
\begin{equation}\label{eq:CZEst}
\left| D^p_s D^{p'}_t K^\pm(s,t) \right|
\leq 
C \frac{d^{p+p'}(p+p')!}{|s-t|^{p+p'+1}} 
\;.
\end{equation}
\end{proposition}
\begin{proof} 
The argument for $K^-$ is elementary, 
due to the explicit relation
\[
D_x^p(1/x) = (-1)^p p! x^{-p-1}.
\]
The argument for $K^+$ follows from this and
the elementary bound
\[
\forall (s,t) \in I^2:\quad
|s-t| \leq s+t \;,
\] 
since then for all $q\in \IN$ holds 
$|s+t|^{-q} \leq |s-t|^{-q}$.
\end{proof}
\begin{remark}\label{rmk:CZStr}
The Calder\'{o}n--Zygmund estimate in \eqref{eq:CZEst} is actually
stronger than what is needed in e.g.\ \cite{DHS1,SchneiderBuch1998}
and the references there to justify the ensuing matrix-compressibility
analysis. 
There, only weaker bounds of the form 
$$
\left| D^p_s D^{p'}_t K^\pm(s,t) \right|
\leq 
C_{pp'} \frac{1}{|s-t|^{p+p'+1}} 
$$
with $0\leq p,p'\leq P(d)$ and 
some (implicit) dependence of the constant
$C_{pp'}$ on $p,p'$ are sufficient for optimal
operator compression rates.

The stronger bounds \eqref{eq:CZEst} correspond to what is 
called in $\cH$-matrix analysis in \cite{Hack1,Hack2} ``asymptotic 
smoothness'' of kernels, implying in particular that also $\cH$-matrix 
formats can be used for approximate representations of $\HT^\pm$ of 
log-linear complexity, see \cite[Section~3.3]{SteinbachZankJNUM2021}.
The analytic derivative bounds \eqref{eq:CZEst} are also essential
for exponentially convergent numerical quadrature methods for 
the numerical evaluation of the nonzero matrix entries.

The explicit expression for the kernel derivatives of $K^-$ implies that 
the order of $|s-t|$ in the kernel estimate \eqref{eq:CZEst} is sharp.
\end{remark}

\subsection{Matrix compression}
\label{sec:MatCmpr}
The basic ingredients in the analysis of the compression 
procedure are estimates for the matrix entries $\langle B\psi_{j',k'},
\psi_{j,k} \rangle_{L^2(I)}$ with $k\in\nabla_j$, $k'\in\nabla_{j'}$
and $j,j'>0$, where $B \colon \, H_{0,}^q(I)\to [H_{0,}^q(I)]'$
is an integral operator of order $2q$ with $q \in [0,1]$. In 
the present application, we are interested in the operators 
$\langle \HT \psi_{j',k'},\partial_t \psi_{j,k} \rangle_{L^2(I)}$, which
is of order $q=1/2$, and $\langle \HT \psi_{j',k'},\psi_{j,k} \rangle_{L^2(I)}$,
which is of order $q=0$. However, our results apply also to the operator 
$\langle \HT \partial_t \psi_{j',k'},\partial_t \psi_{j,k}\rangle_{L^2(I)}$,
which is of order $q = 1$ and which would appear in the discretization 
of the wave equation.

The supports of the wavelets will be denoted by
\begin{equation}	\label{eq:def supp}
   \Xi_{j,k}  = \operatorname{supp}(\psi_{j,k}) \subset \overline{I}.
\end{equation}
Moreover, the so-called ``second compression''
involves the {\emph singular support\/} of a wavelet $\psi_{j,k}$,
\begin{equation}	\label{eq:def supp'}
   \Xi_{j,k} ^s = \operatorname{sing\;supp}\psi_{j,k}
   	\subset \Xi_{j,k} .
\end{equation}
It consists of all points of the wavelets' support, 
where $\psi_{j,k}$ is not smooth.
A proof of the following estimates for the matrix entries
is based on the Calder\'{o}n--Zygmund estimates \eqref{eq:CZEst} 
and can be found e.g.~in \cite[Chap.~8]{SchneiderBuch1998},
see also \cite{RSt2003} for related results.
\begin{theorem}  \label{thm:1st estimate}
Suppose $j,j'\ge 0$ and let $k\in\nabla_j$, $k'\in\nabla_{j'}$ 
be such that
\begin{equation*}
  \operatorname{dist}(\Xi_{j,k} ,\Xi_{j',k'} ) > 0.
\end{equation*}
Then, one has
\[
  |\langle B\psi_{j',k'},\psi_{j,k}\rangle_{L^2(I)}|
  \lesssim 
  \frac{2^{-(j+j')(\widetilde{d}+1/2)}}
	{\operatorname{dist}(\Xi_{j,k} ,\Xi_{j',k'} )^{1+2q+2\widetilde{d}}}
\]
uniformly with respect to $j$ and $j'$.
\end{theorem}

\begin{theorem}  \label{thm:2nd estimate}
Suppose that $j' > j \ge 0$. Then, for $k\in\nabla_j$, $k'\in\nabla_{j'}$, 
the coefficients $\langle B\psi_{j',k'},\psi_{j,k}\rangle_{L^2(I)}$ and
$\langle B\psi_{j,k},\psi_{j',k'}\rangle_{L^2(I)}$ satisfy
\[
  |\langle B\psi_{j',k'},\psi_{j,k}\rangle_{L^2(I)}|,\
  |\langle B\psi_{j,k},\psi_{j',k'}\rangle_{L^2(I)}|
  \lesssim \frac{2^{j/2}2^{-j'(\widetilde{d}+1/2)}}
	{\operatorname{dist}(\Xi_{j,k} ^s,\Xi_{j',k'} )^{2q+\widetilde{d}}}
\]
uniformly with respect to $j$ and $j'$, provided that
$\operatorname{dist}(\Xi_{j,k} ^s,\Xi_{j',k'} ) \gtrsim 2^{-j'}$.
\end{theorem}

Based on Theorems \ref{thm:1st estimate} and 
\ref{thm:2nd estimate}, we arrive at the following 
theorem, which defines a compressed version of the matrix 
${\bf B}_j = [\langle B\phi_{j,k'},\phi_{j,k}\rangle_{L^2(I)}]_{k,k'\in\Delta_j}$
when expressed in wavelet coordinates.

\begin{theorem}[A-priori compression \cite{DHS1}]\label{thm:a-priori}
Let the supports $\Xi_{j,k}$ and $\Xi_{j,k}^s$ be given as in
\eqref{eq:def supp} and \eqref{eq:def supp'} and define the 
compressed system matrix ${\bf B}_j^c$, corresponding to 
the discretization of the integral operator $B$ with respect to 
the trial space $S_j$, by
\begin{equation}		\label{eq:a-priori}
  [{\bf B}_j^c]_{(\ell,k),(\ell,',k')}
  = \begin{cases}
  \qquad \quad 0,& \operatorname{dist}(\Xi_{\ell,k} ,\Xi_{\ell',k'} ) 
  	> \mathcal{B}_{\ell,\ell'}\ \text{and $\ell,\ell'\ge 0$}, \\
  \qquad \quad 0,& \operatorname{dist}(\Xi_{\ell,k} ,\Xi_{\ell',k'} ) 
  	\le 2^{-\min\{\ell,\ell'\}}\ \text{and}\\
                 &\operatorname{dist}(\Xi_{\ell,k} ^s,\Xi_{\ell',k'} ) >
		 \mathcal{B}_{\ell,\ell'}^s\ \text{if $\ell' > \ell\geq 0$}, \\
		 & \operatorname{dist}(\Xi_{\ell,k} ,\Xi_{\ell',k'} ^s) >
		 \mathcal{B}_{\ell,\ell'}^s\ \text{if $\ell > \ell'\geq 0$}, \\
  \langle B\psi_{\ell',k'},\psi_{\ell,k}\rangle_{L^2(I)}, &\text{otherwise}.
  \end{cases}
\end{equation}
For given, fixed parameters 
\begin{equation*}
  a > 1, \qquad d < \delta < \widetilde{d} + 2q,
\end{equation*}
for all $\ell,\ell' \geq 0$ 
the matrix compression parameters 
$\mathcal{B}_{\ell,\ell'}$ 
and 
$\mathcal{B}_{\ell,\ell'}^s$ 
are set according to
\begin{equation}	\label{eq:cut-off parameters}
  \begin{aligned}
  \mathcal{B}_{\ell,\ell'} = a \phantom{'}\max\left\{ 2^{-\min\{\ell,\ell'\}},
	2^{\frac{2j(\delta-q)-(\ell+\ell')(\delta+\widetilde{d})}{2(\widetilde{d}+q)}}\right\}, \\
  \mathcal{B}_{\ell,\ell'}^s = a \max\left\{ 2^{-\max\{\ell,\ell'\}},
	2^{\frac{2j(\delta-q)-(\ell+\ell')\delta-\max\{\ell,\ell'\}
	\widetilde{d}}{\widetilde{d}+2q}}\right\}.
  \end{aligned}
\end{equation}
With the compression according to \eqref{eq:a-priori}--\eqref{eq:cut-off parameters}, 
the compressed system matrices ${\bf B}_j^c$ have only $\mathcal{O}(N_j)$ 
nonzero entries. 
In addition, the error estimates
\begin{equation*}
    \|u-u_j\|_{H^{q-t}(I)} \lesssim 2^{j(d-q+t)}\|u\|_{H^d(I)},
    	\quad \text{for all $0\le t\le d-q$}
\end{equation*}
hold for the solution $u_j$ of the compressed Galerkin 
system ${\bf B}_j^c{\bf u}_j = {\bf f}_j$ provided that 
$u = B^{-1}f \in H^d(I)$.
\end{theorem}

\section{Numerical experiments}
\label{Sec:NumExp}

\subsection{1D test example}
\label{sec:NumExp1d}
In order to demonstrate the wavelet method and to
validate its accuracy, we consider the ordinary differential equation 
\begin{equation}\label{eq:ODE}
  \partial_t u(t) + \mu u(t) = f(t)\quad\text{for $t\in I=(0,T)$},\quad u(0) = 0
\end{equation}
and choose $T=2$, $\mu=10$, and the right-hand side $f(t)$ such that
\[
  u(t) = -2\sin\bigg(\frac{3\pi}{4}t\bigg) + \sin\bigg(\frac{9\pi}{4}t\bigg)
\]
becomes the known solution to be determined.

\begin{table}[hbt]
\begin{tabular}{|c|c|c|cc|cc|cc|}\hline
$j$ & $N_j^t$ & \% & \multicolumn{2}{c|}{$L^2(I)$-error} 
& \multicolumn{2}{c|}{$H^1_{0,}(I)$-error}
& \multicolumn{2}{c|}{$H^{1/2}_{0,}(I)$-error} \\ \hline
4 & 16 & 98.62 & $3.28\cdot 10^{-2}$ &          & 1.88  &            & $2.48\cdot 10^{-1}$ &          \\ 
5 & 32 & 84.57 & $7.64\cdot 10^{-3}$ & (2.10) & $9.28\cdot 10^{-1}$ & (1.02) & $8.42\cdot 10^{-2}$ & (1.56) \\
6 & 64 & 62.32 & $1.87\cdot 10^{-3}$ & (2.03) & $4.62\cdot 10^{-1}$ & (1.00) & $2.94\cdot 10^{-2}$ & (1.52) \\
7 & 128 & 42.09 & $4.67\cdot 10^{-4}$ & (2.01) & $2.31\cdot 10^{-1}$ & (1.00) & $1.04\cdot 10^{-2}$ & (1.50) \\
8 & 256 & 26.76 & $1.17\cdot 10^{-4}$ & (2.00) & $1.15\cdot 10^{-1}$ & (1.00) & $3.67\cdot 10^{-3}$ & (1.50) \\
9 & 512 & 16.27 & $2.91\cdot 10^{-5}$ & (2.00) & $5.77\cdot 10^{-2}$ & (1.00) & $1.30\cdot 10^{-3}$ & (1.50) \\
10 & 1024 & 9.58 & $7.28\cdot 10^{-6}$ & (2.00) & $2.89\cdot 10^{-2}$ & (1.00) & $4.58\cdot 10^{-4}$ & (1.50) \\
11 & 2048 & 5.50 & $1.82\cdot 10^{-6}$ & (2.00) & $1.44\cdot 10^{-2}$ & (1.00) & $1.62\cdot 10^{-4}$ & (1.50)\\
12 & 4096  & 3.09 & $4.55\cdot 10^{-7}$ & (2.00) & $7.21\cdot 10^{-3}$ & (1.00) & $5.73\cdot 10^{-5}$ & (1.50)\\
13 & 8192 & 1.71 & $1.14\cdot 10^{-7}$ & (2.00) & $3.61\cdot 10^{-3}$ & (1.00) & $2.03\cdot 10^{-5}$ & (1.50)\\\hline
\end{tabular}
\caption{\label{tab:1DA}
Results for the analytical example \eqref{eq:ODE} 
on the interval for piecewise linear wavelets 
with $\tilde{d}=2$ vanishing moments.}
\end{table}

\begin{table}[hbt]
\begin{tabular}{|c|c|c|cc|cc|cc|}\hline
$j$ & $N_j^t$ & \% & \multicolumn{2}{c|}{$L^2(I)$-error} 
& \multicolumn{2}{c|}{$H^1_{0,}(I)$-error}
& \multicolumn{2}{c|}{$H^{1/2}_{0,}(I)$-error} \\ \hline
5 & 32 & 96.51 & $7.63\cdot 10^{-3}$ &          & $9.28\cdot 10^{-1}$ &             & $8.42\cdot 10^{-2}$ &         \\
6 & 64 & 78.37 & $1.87\cdot 10^{-3}$ & (2.03) & $4.62\cdot 10^{-1}$ & (1.00) & $2.94\cdot 10^{-2}$ & (1.52) \\
7 & 128 & 56.14 & $4.66\cdot 10^{-4}$ & (2.01) & $2.31\cdot 10^{-1}$ & (1.00) & $1.04\cdot 10^{-2}$ & (1.50) \\
8 & 256 & 37.19 & $1.16\cdot 10^{-4}$ & (2.00) & $1.15\cdot 10^{-1}$ & (1.00) & $3.67\cdot 10^{-3}$ & (1.50) \\
9 & 512 & 23.22 & $2.91\cdot 10^{-5}$ & (2.00) & $5.77\cdot 10^{-2}$ & (1.00) & $1.30\cdot 10^{-3}$ & (1.50) \\
10 & 1024 & 13.90 & $7.28\cdot 10^{-6}$ & (2.00) & $2.89\cdot 10^{-2}$ & (1.00) & $4.58\cdot 10^{-4}$ & (1.50) \\
11 & 2048 & 8.06 & $1.82\cdot 10^{-6}$ & (2.00) & $1.44\cdot 10^{-2}$ & (1.00) & $1.62\cdot 10^{-4}$ & (1.50)\\
12 & 4096 & 4.57 & $4.55\cdot 10^{-7}$ & (2.00) & $7.21\cdot 10^{-3}$ & (1.00) & $5.73\cdot 10^{-5}$ & (1.50)\\
13 & 8192 & 2.54 & $1.14\cdot 10^{-7}$ & (2.00) & $3.61\cdot 10^{-3}$ & (1.00) & $2.03\cdot 10^{-5}$ & (1.50)\\\hline
\end{tabular}
\caption{\label{tab:1DB}
Results for the analytical example \eqref{eq:ODE} 
on the interval for piecewise linear wavelets
with $\tilde{d}=4$ vanishing moments.}
\end{table}

\begin{figure}[htb]
\begin{center}
\pgfplotsset{width=0.60\textwidth, height=0.50\textwidth}
\begin{tikzpicture}
\begin{axis}[grid, ymin= 0, ymax = 60, xmin = 0, xmax = 4100, ytick={10,30,50},
	legend style={at={(0.98,0.68)},anchor=east},
	xtick={0,1000,2000,3000,4000}, ylabel={condition number}, xlabel={degrees of freedom}]
\addplot[line width=0.7pt,color=red,mark=*] table[x=dof,y=cond1]{
dof cond10 cond1 cond100
 8  4.2814 4.8998 16.9206
 16  5.3364 6.3900 18.5658
 32  6.9383 7.5467 20.7352
 64  8.5190 8.4374 25.1417
 128  9.9160 9.1268 31.7880
 256  11.0957 9.6654 38.9254
 512  12.0713 10.0908 45.5147
 1024  12.8703 10.4309 51.2215
 2048  13.5223 10.7056 56.0126
 4096  14.0545 10.9300 59.9721
};
\addlegendentry{$\mu=1$};
\addplot[line width=0.7pt,color=blue,mark=*] table[x=dof,y=cond10]{
dof cond10 cond1 cond100
8  4.2814 4.8998 16.9206
 16  5.3364 6.3900 18.5658
 32  6.9383 7.5467 20.7352
 64  8.5190 8.4374 25.1417
 128  9.9160 9.1268 31.7880
 256  11.0957 9.6654 38.9254
 512  12.0713 10.0908 45.5147
 1024  12.8703 10.4309 51.2215
 2048  13.5223 10.7056 56.0126
 4096  14.0545 10.9300 59.9721
 };
\addlegendentry{$\mu=10$};
\addplot[line width=0.7pt,color=dgreen,mark=*] table[x=dof,y=cond100]{
dof cond10 cond1 cond100
 8  4.2814 4.8998 16.9206
 16  5.3364 6.3900 18.5658
 32  6.9383 7.5467 20.7352
 64  8.5190 8.4374 25.1417
 128  9.9160 9.1268 31.7880
 256  11.0957 9.6654 38.9254
 512  12.0713 10.0908 45.5147
 1024  12.8703 10.4309 51.2215
 2048  13.5223 10.7056 56.0126
 4096  14.0545 10.9300 59.9721
};
\addlegendentry{$\mu=100$};
\end{axis}
\end{tikzpicture}
\caption{\label{fig:cond}Condition numbers of the diagonally 
scaled system matrix ${\bf A}_j^t+\mu {\bf M}_j^t$ in wavelet 
coordinates for $\mu = 1, 10, 100$.}
\end{center}
\end{figure}
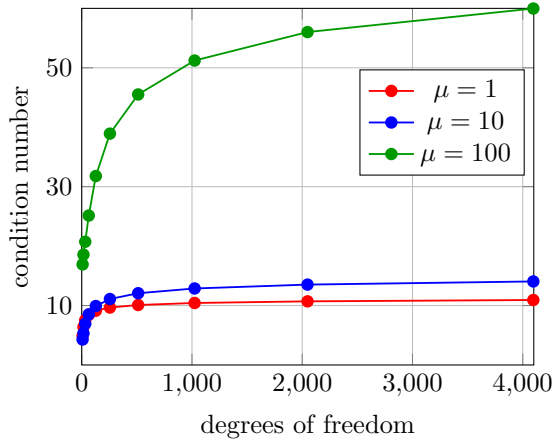

In Table~\ref{tab:1DA} and Table~\ref{tab:1DB}, we find the results 
when discretizing \eqref{eq:ODE} by a Galerkin method using wavelet 
matrix compression with $\tilde{d} = 2$ and $\tilde{d} = 4$ vanishing 
moments, respectively. The respective compression rates for the 
associated system matrices ($\operatorname{nnz}({\bf A}_j^t+\mu 
{\bf M}_j^t)/N^2$ in percent) are found in the column entitled ``\%''. 
Although the system matrices are quite sparse, we achieve the 
optimal rate of convergence of order 2 in $L^2(I)$, of order 1 in
$H^1_{0,}(I)$, and of order $3/2$ in $H^{1/2}_{0,}(I)$.
In particular, the accuracy of the compressed scheme
is the same as that of the original scheme.
For the error in $H^{1/2}_{0,}(I)$, we use the right side of the interpolation inequality
\begin{equation*}
   \forall v \in H^1_{0,}(I):  \| v \|_{H^{1/2}_{0,}(I)} \leq C \sqrt{ \| v \|_{L^2(I)} \cdot \| \partial_t v \|_{L^2(I)} }
\end{equation*}
with a constant $C>0$, see \cite[Proposition~2.3, p. 19]{LM1}. In other words, in the last column of Table~\ref{tab:1DA} and Table~\ref{tab:1DB}, the quantity $\sqrt{ \| \cdot \|_{L^2(I)} \cdot \| \partial_t \cdot \|_{L^2(I)} }$ is treated as $\| \cdot \|_{H^{1/2}_{0,}(I)}$.

\subsection{Preconditioning of space-time methods}
It is well-known that proper scaling of wavelet bases 
leads to norm equivalences \cite{DA1,DK92,Jaff} for 
a whole scale of Sobolev spaces. Indeed, a diagonal 
scaling of matrices in wavelet coordinates yields uniformly 
bounded condition numbers provided that the underlying 
operators are elliptic and continuous with respect to the 
given Sobolev space provided that the wavelets and their 
duals are regular enough. 

The condition numbers of the diagonally scaled system 
matrix ${\bf A}_j^t+\mu {\bf M}_j^t$ from the previous 
example are shown in Figure~\ref{fig:cond} for $\widetilde{d}
= 2$ and different values of the coupling parameter $\mu$. 
The condition numbers appear indeed uniformly bounded 
with respect to $h$ as the matrix ${\bf A}_j^t$ is equivalent 
to the $H_{0,}^{1/2}(I)$-norm and ${\bf M}_j^t$ is positive 
and bounded by the $L^2$-norm. However, we observe 
a strong dependence on the parameter $\mu$. In fact, 
the diagonally scaled version of ${\bf A}_j^t$ is also 
uniformly bounded, but the condition number of the 
matrix ${\bf M}_j^t$ grows linearly in $1/h_j^t$. 
Diagonal scaling does not help here and makes the 
situation even worse. Therefore, (standard) multilevel 
preconditioning for space-time formulations will not work
since we have to precondition -- in case of the heat 
equation -- the system matrix
\begin{equation}\label{eq:precond1}
  {\bf A}_j^t \otimes {\bf M}_j^x + {\bf M}_j^t\otimes {\bf A}_j^x
\end{equation}
with ${\bf M}_j^x$ being the spatial mass matrix and ${\bf A}_j^x$
corresponding to the spatial Laplacian.

In \cite{LZ21}, it was proposed to use the Bartels--Stewart method 
for the efficient solution of \eqref{eq:precond1}. This, however, requires
the Schur decomposition of the matrix $\big({\bf A}_j^t\big)^{-1}{\bf M}_j^t$
which is computationally expensive and hence time-consuming. We 
shall therefore make use of the observation from \cite{HM21} that the 
LU factorization of the compressed wavelet matrices can be computed 
very efficiently by means of nested dissection. Figure~\ref{fig:dissect} 
depicts the sparsity pattern of the compressed system matrix 
${\bf A}_j^t+\mu {\bf M}_j^t$ in wavelet coordinates. The fill-in of 
the LU factorization is minimal as both, the left factor as well as the 
right factor, have fewer nonzero entries than the original matrix. As a 
consequence, we can apply a fast sparse direct
solver for the linear systems in the temporal coordinate. 

\begin{figure}
\begin{center}
\includegraphics[trim=90 10 90 20,clip,width=0.45\textwidth,]{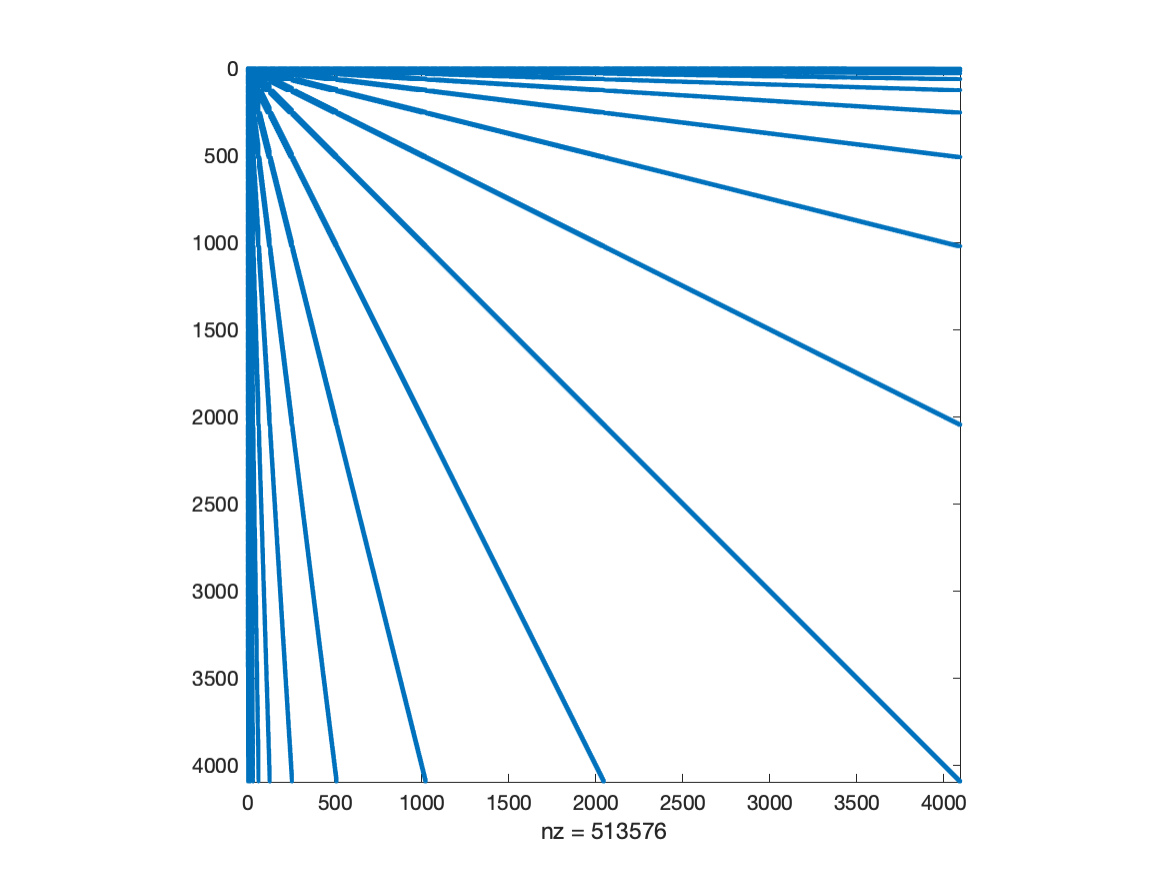}
\includegraphics[trim=90 10 90 20,clip,width=0.45\textwidth]{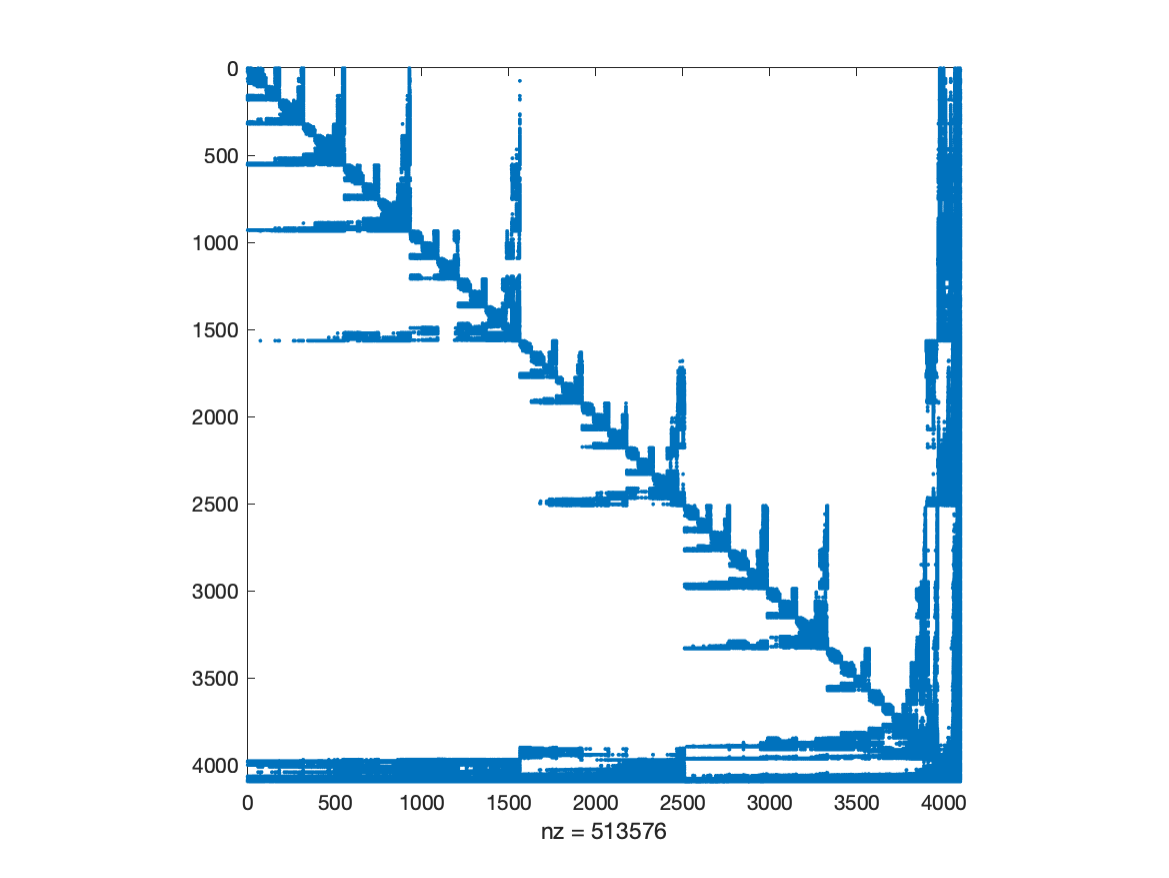}\\
\includegraphics[trim=90 10 90 20,clip,width=0.45\textwidth]{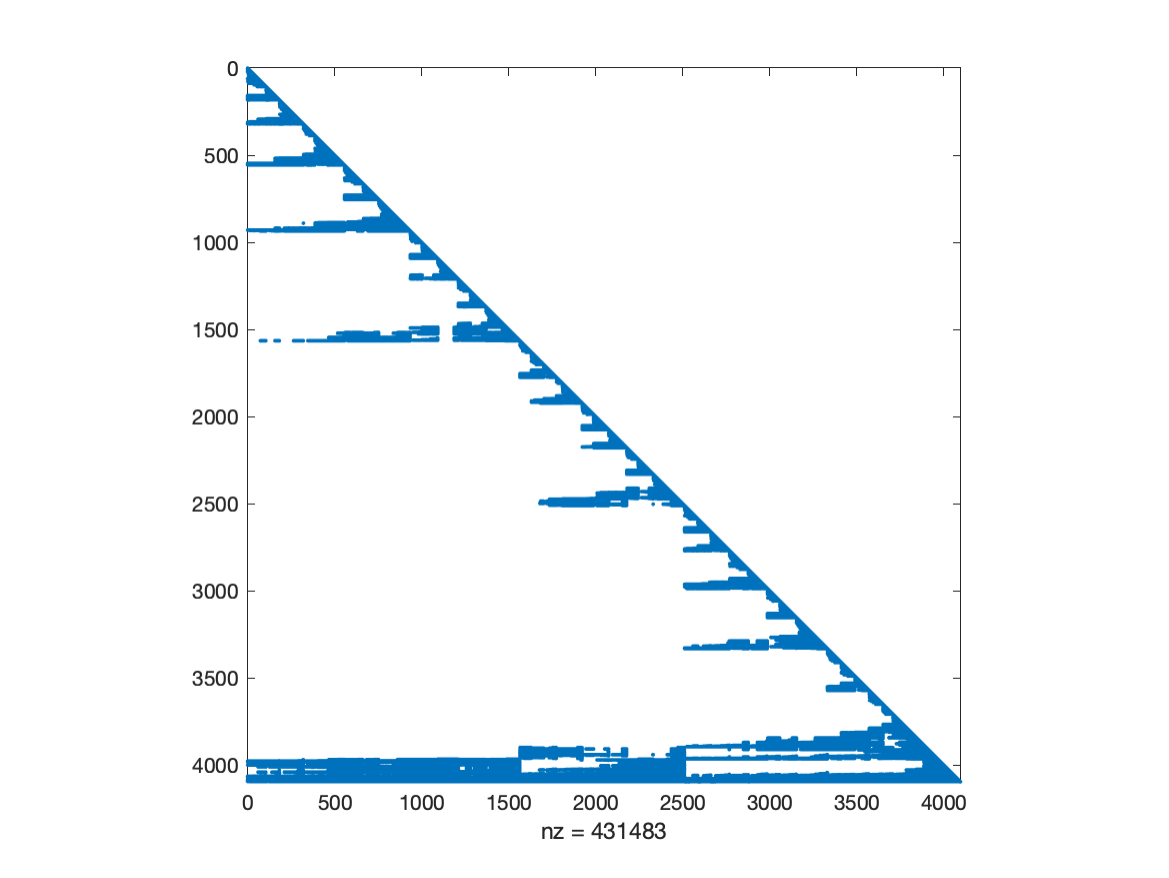}
\includegraphics[trim=90 10 90 20,clip,width=0.45\textwidth]{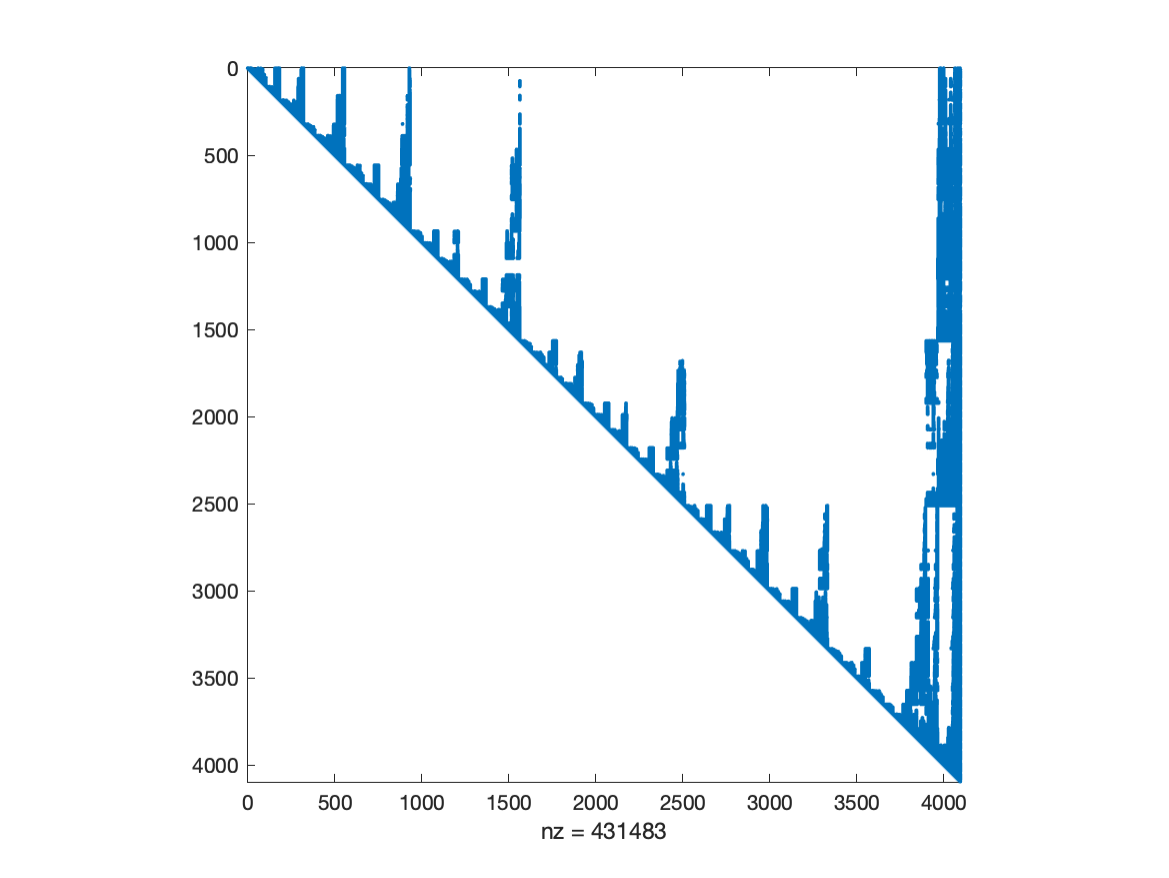}
\caption{\label{fig:dissect}
Sparsity patterns of the compressed system matrix (top left), 
its reordered version (top right) and of the associated LU factors
(bottom).}
\end{center}
\end{figure}

In order to exploit this fact, we assume that the
sequence 
\[
  S_0^x\subset S_1^x\subset S_2^x\subset \cdots,
  	\quad\text{where}\ S_j^x = \spn\{\phi_{j,k}^x:k\in\Delta_j^x\}
\]
consists of Lagrangian finite element spaces of continuous, 
piecewise polynomials of total degree $p_x\geq 1$ 
on a nested sequence of regular, 
simplicial quasiuniform partitions of the spatial domain $\domain$ that
is 
obtained by a systematic refinement, 
for example by the newest edge bisection or dyadic subdivision, 
which we will use here. 
Here, $\Delta_j^x = \{1,\ldots,N_j^x\}$ 
denotes the set of indices for (one-scale) nodal basis functions of $S^x_j$.
For preconditioning the system matrix \eqref{eq:precond1}, 
we 
switch to multilevel frame coordinates of BPX-type
in $S^x_j(\domain)$, see \cite{BPX, Griebel95,HSS08} for details. 
Then, on each grid-scale in the spatial domain $\domain$,
we have to perform a diagonal scaling of the system matrix 
\eqref{eq:precond1} for the current spatial meshwidth $h_j^x$. 
Since the entries of the finite element mass matrix $\diag({\bf M}_j^x)$ 
scale asymptotically like $1$ and those of the stiffness 
matrix $\diag({\bf A}_j^x)$ like $(h_j^x)^{-2}$ provided 
that the finite elements in space are $L^2$-normalized, 
we need to apply the diagonal scaling to
\[
  {\bf A}_j^t \otimes \diag({\bf M}_j^x) + {\bf M}_j^t\otimes \diag({\bf A}_j^x)
  	\sim \big({\bf A}_j^t+(h_j^x)^{-2}{\bf M}_j^t\big) \otimes {\bf I}_j^x,
\]
where ${\bf I}_j^x$ denotes the identity matrix in
$\mathbb{R}^{N_j^x\times N_j^x}$. 
Since the inverse of 
${\bf A}_j^t+(h_j^x)^{-2}{\bf M}_j^t$ can be cheaply 
computed, this diagonal scaling is feasible. 
The 
implementation of this preconditioner is along the 
lines of \cite{BPX}. Denoting the (spatial) prolongation 
from level $\ell$ to level $j>\ell$ by ${\bf I}_{\ell}^j$ and 
the (spatial) restriction from level $j$ to level $\ell<j$ 
by ${\bf I}_j^\ell$, one can implement the multilevel 
preconditioner in accordance with
\begin{equation}\label{eq:BPX}
 \sum_{\ell=0}^j \big({\bf I}_j^t\otimes {\bf I}_\ell^j\big) 
 \Big(\big({\bf A}_j^t+(h_\ell^x)^{-2}{\bf M}_j^t\big)^{-1}\otimes {\bf I}_\ell^x\Big)
 	\big({\bf I}_j^t\otimes{\bf I}_j^\ell\big).
\end{equation}
Here, in analogy to ${\bf I}_j^x$ introduced above, ${\bf I}_j^t$ 
denotes the identity matrix in $\mathbb{R}^{N_j^t\times N_j^t}$.
Note that the preconditioner \eqref{eq:BPX} is optimal 
\emph{except 
for a logarithmic factor which appears since the multilevel
frame is not stable in $L^2(\domain)$}, 
compare~\cite{HSS08} for example.

\subsection{Heat equation in two spatial dimensions}
\label{sec:Heat2D}
Let $\domain$ be the unit square $(0,1)^2$ and consider the
heat equation \eqref{Einf:PDG} with right-hand side
\[
f(x,t) = \sin(2\pi x_1)\sin(2\pi x_2)\big(\cos(t)+ 2\pi^2\sin(t)\big).
\] 
This yields the known solution $u(x,t) = \sin(2\pi x_1)
\sin(2\pi x_2)\sin(t)$ of the heat equation. 

\begin{table}
\begin{center}
\begin{tabular}{|c|c|c|cc|c|cc|}\hline
\multicolumn{2}{|c|}{unknowns} & \multicolumn{3}{c}{full tensor-product space} &
\multicolumn{3}{|c|}{sparse tensor-product space} \\ \hline
$N_j^t$ & $N_j^x$ & iter & \multicolumn{2}{c|}{$L^2(Q)$-error} 
& iter & \multicolumn{2}{c|}{$L^2(Q)$-error} \\\hline
16 & 289 & 19 & $3.0029\cdot 10^{-3}$ & & 22 & $3.0022\cdot 10^{-3}$ & \\
32 & 1089 & 20 & $7.5041\cdot 10^{-4}$ & (2.00) & 37 & $7.5041\cdot 10^{-4}$ & (2.00) \\
64 & 4225 & 20 & $1.8752\cdot 10^{-4}$ & (2.00) & 45 & $1.8752\cdot 10^{-4}$ & (2.00) \\
128 & 16641 & 20 & $4.6868\cdot 10^{-5}$ & (2.00) & 40 & $4.6868\cdot 10^{-5}$ & (2.00) \\
256 & 66049 & 20 & $1.1715\cdot 10^{-5}$ & (2.00) & 37 & $1.1715\cdot 10^{-5}$ & (2.00) \\
512 & 263169 & 21 & $2.9284\cdot 10^{-6}$ & (2.00) & 39 & $2.9285\cdot 10^{-6}$ & (2.00) \\
1024 & 1050625 & 21 & $7.3195\cdot 10^{-7}$ & (2.00) & 35 & $7.3226\cdot 10^{-7}$ & (2.00) \\\hline
\end{tabular}
\caption{\label{tab:precond}
Approximation error and number of iterations needed to solve the 
linear system of equations up to a relative error of $10^{-8}$ in $2+1$ 
dimensions. The table contains the results of both, the full and 
the sparse tensor-product discretization.}
\end{center}
\end{table}

On the level $j$, we subdivide the unit square into 
$2^j \times 2^j$ squares and discretize in the
spatial coordinate by piecewise bilinear finite elements. 
In the temporal coordinate, we apply piecewise linear 
wavelets with four vanishing moments. For the solution
of the linear system of equations, we use the iterative 
Krylov subspace solver GMRES (generalized minimal 
residual method, see \cite{SS86}), preconditioned by 
the multilevel preconditioner proposed in the previous 
subsection. We started the iteration with the initial guess
${\bf 0}$ and stopped it when the relative norm of the 
residual is smaller than $10^{-8}$.

The observed numerical results for the levels $j=4,\ldots,10$
are reported in Table~\ref{tab:precond} in the columns 
entitled ``full tensor-product space''. 
Therein, 
the number of iterations needed by the preconditioned 
GMRES solver appears to be essentially bounded. 
The expected rate $4^{-j}$ of convergence is 
realized.

\section{Sparse space-time tensor-product approximation}
\label{sec:SG}
In this section, we introduce a sparse space-time 
tensor-product Galerkin approximation to the solution 
of the parabolic evolution equation \eqref{Einf:PDG}, 
including space-time error estimates.

\subsection{Sparse tensor-product spaces}
\label{sec:SpTPSpc}
For a given refinement level $j \in \IN_0$, 
instead of using the full tensor-product space
$S_j^{x,t} = S_j^x\otimes S_j^t \subset H^{1,1/2}_{\Gamma_\mathrm{D};0,}(Q)$, we can also consider 
the \emph{sparse tensor-product space} $\widehat{S}_j^{x,t}$
for the discretization of the heat equation. In view of the multilevel 
decomposition $S_j^t = \bigoplus_{\ell=0}^j W_\ell^t$ in the temporal variable,
the sparse tensor-product space is defined by
\begin{equation}\label{eq:sparseTPS}
\widehat{S}_j^{x,t} = \bigoplus_{\ell=0}^j S_\ell^x \otimes W_{j-\ell}^t \subset H^{1,1/2}_{\Gamma_\mathrm{D};0,}(Q).
\end{equation}
With the representations $S_j^x = \spn\{\phi_{j,k}^x:k\in\Delta_j^x\}$ 
and $W_j^t = \spn\{\psi_{j,k}^t:k\in\nabla_j^t\}$, a 
basis in the sparse tensor-product space is given by 
\begin{equation}\label{eq:sparseTPbasis}
\widehat{S}_j^{x,t} = \spn\{\phi_{\ell,k}^x \otimes \psi_{j-\ell,k'}^t:
k\in\Delta_\ell^x,\ k'\in\nabla_{j-\ell}^t,\ \ell=0,\ldots,j\}.
\end{equation}
We refer to \cite[Section 6.1]{GH13} for all the details.

\subsection{Sparse tensor space-time Galerkin projection}
\label{sec:xtGalPrj}
The dimension of the sparse tensor-product space $\widehat{S}_j^{x,t}$
scales in general linearly in the maximum of the number 
of degrees of freedom in space and time, i.e. it scales as
$\mathcal{O}(\max\{N_j^x,N_j^t\})$, 
with an additional logarithm appearing
if $N_j^x = N_j^t$, compare \cite[Theorem 4.1]{GH13}. 
In \cite{GH13} also more general settings 
of the sparse tensor-product compared 
to the setting \eqref{eq:sparseTPS} have been studied. 

The \emph{sparse space-time tensor-product Galerkin 
approximation}
is to find $\widehat u_j \in \widehat{S}_j^{x,t} \subset 
H^{1,1/2}_{\Gamma_\mathrm{D};0,}(Q)$ such that
\begin{equation} \label{Sparse:WaermeFEM}
  \forall \widehat v_j \in \widehat{S}_j^{x,t}: \quad 
  b(\widehat u_j, \HT \widehat v_j)= \spf{f}{\HT \widehat v_j}_{Q}.
\end{equation}
The sparse space-time tensor-product approximation 
scheme~\eqref{Sparse:WaermeFEM} is well-defined and 
unconditionally stable with the space-time stability estimate
\[
  \norm{\widehat u_j}_{ H^{1/2}_{0,}(I;L^2(\domain))} 
  \leq \norm{f}_{[H^{1/2}_{,0}(I;L^2(\domain))]'},
\]
provided that $f \in [H^{1/2}_{,0}(I;L^2(\domain))]'$, 
see \cite[Theorem~3.4.20]{ZankDissBuch2020}. This is 
essential to ensure the stability of the projection defined by
\eqref{Sparse:WaermeFEM}, as in the sparse tensor-product 
\eqref{eq:sparseTPS} with basis \eqref{eq:sparseTPbasis} 
all possible combinations of spatial and temporal meshwidths 
appear simultaneously.

\subsection{Convergence rates}
\label{sec:CnvRt}
In this subsection, we prove convergence of the sparse tensor-product 
approximation $\widehat u_j$ given by~\eqref{Sparse:WaermeFEM}.

\begin{theorem}  \label{thm:SpxtErrBd}
Let $j \in \IN_0$ be a given refinement level
and let $u \in H^{1,1/2}_{\Gamma_\mathrm{D};0,}(Q)$ 
be the unique solution of the space-time variational 
formulation~\eqref{eq:IBVPVar}.
Let $\widehat u_j \in \widehat{S}_j^{x,t}$ 
be the sparse tensor-product 
approximation of~\eqref{eq:IBVPVar} 
given by~\eqref{Sparse:WaermeFEM}, 
and assume that $\mathcal L_x u \in L^2(Q)$ holds.

Then, 
the space-time error estimate
\begin{equation}\label{eq:SGerror}
\begin{aligned}
 &\| u - \widehat u_j \|_{H^{1/2}_{0,}(I; L^2(\domain))}
  \le 
   2\| u - \widehat P_j u \|_{H^{1/2}_{0,}(I; L^2(\domain))}\\
   &\qquad+ \big\| (\mathrm{Id}_{xt} - P_j^t \otimes \mathrm{Id}_x) \mathcal L_x u \big\|_{[H^{1/2}_{,0}(I; L^2(\domain))]'}  \\
   &\qquad+ C \sum_{\ell=0}^j (h_\ell^x)^{-1} \big\|\big((P_j^t-P_{j-\ell}^t)
    	\otimes (P_{\ell}^x - P_{\ell-1}^x)\big) u\big\|_{H^1(\domain) \otimes [H^{1/2}_{0,}(I)]'}.
\end{aligned}
\end{equation}
holds true, with a constant $C>0$ and 
the sparse space-time tensor-product projection defined by
  \begin{equation*}
    \widehat P_j u 
    = 
    \sum_{\ell=0}^j \big(P_{j-\ell}^t\otimes (P_{\ell}^x - P_{\ell-1}^x)\big) u \in \widehat{S}_j^{x,t}.
  \end{equation*}
  Here, for $\ell \in \IN_0$, $P_\ell^t\colon L^2(I)\to S_\ell^t$ 
  is any temporal projection and $P_\ell^x \colon \, 
  H^1_{\Gamma_\mathrm{D}}(\domain) \to S_\ell^x$ 
  denotes the spatial Ritz projection as defined in 
  Theorem~\ref{thm:convergence} 
  Further, we set $P_{-1}^x = 0$.
\end{theorem}

\begin{proof}
  The triangle inequality yields
  \begin{equation*}
    \| u - \widehat u_j \|_{H^{1/2}_{0,}(I; L^2(\domain))} 
    \leq \| u - \widehat P_j u \|_{H^{1/2}_{0,}(I; L^2(\domain))} 
    + 
    \| \widehat u_j - \widehat P_j u \|_{H^{1/2}_{0,}(I; L^2(\domain))}.
  \end{equation*}
  The second term in the upper bound is estimated as follows. 

  In view of the $H^{1/2}_{0,}(I; L^2(\domain))$-ellipticity of the 
  bilinear form $b(\cdot,\cdot)$ in \eqref{b_H12elliptisch} and the 
  Galerkin orthogonality corresponding to \eqref{Sparse:WaermeFEM}, we calculate
\begin{equation*}
\begin{array}{rl}
    \| \widehat u_j - \widehat P_j u \|_{H^{1/2}_{0,}(I; L^2(\domain))}^2 
    & \leq 
    b(\widehat u_j - \widehat P_j u, \HT (\widehat u_j - \widehat P_j u )) \\
    & = b( u - \widehat P_j u, \HT (\widehat u_j - \widehat P_j u)) \\
    & = \langle \partial_t (u- \widehat P_j u), \HT (\widehat u_j - \widehat P_j u) \rangle_Q \\
    &+ \langle A \nabla_x (\mathrm{Id}_{xt}-P_j) u, \nabla_x \HT (\widehat u_j - \widehat P_j u)\rangle_{L^2(Q)^n} \\
    &+ \langle A \nabla_x (P_j -\widehat P_j) u, \nabla_x \HT (\widehat{u}_j-\widehat P_j u)\rangle_{L^2(Q)^n}
\end{array}
\end{equation*}
with $P_j u = (P_j^t \otimes P_j^x) u \in S_j^{x,t}$.
For the first term, we have with \eqref{HT:Beschr} that
\[
\langle \partial_t (u- \widehat P_j u), \HT (\widehat u_j - \widehat P_j u) \rangle_Q 
\le\|u- \widehat P_j u\|_{H^{1/2}_{0,}(I; L^2(\domain))}
\|\widehat u_j - \widehat P_j u\|_{H^{1/2}_{0,}(I; L^2(\domain))},
\]
while we find for the second term by the definition of the Ritz projection $P_j^x$ and by integration by parts
\begin{align*}
  &\langle A \nabla_x (\mathrm{Id}_{xt}-P_j) u, 
  	\nabla_x \HT (\widehat u_j - \widehat P_j u)\rangle_{L^2(Q)^n} \\
  &\qquad = \langle A \nabla_x (\mathrm{Id}_{xt}-P_j^t \otimes \mathrm{Id}_x) u, 
  	\nabla_x \HT (\widehat u_j - \widehat P_j u)\rangle_{L^2(Q)^n} \\
  &\qquad \le\big\| (\mathrm{Id}_{xt} - P_j^t \otimes \mathrm{Id}_x) 
  	\mathcal L_x u \big\|_{[H^{1/2}_{,0}(I; L^2(\domain))]'}
  		\|\widehat u_j - \widehat P_j u\|_{H^{1/2}_{0,}(I; L^2(\domain))}.
\end{align*}
The last term is treated as follows. We set 
$\widehat{v}_j = \widehat u_j - \widehat P_j u \in\widehat{S}_j^{x,t}$
and insert the definition of the sparse grid projection 
and the representation $P_j = \sum_{\ell=0}^j P_j^t \otimes (P_\ell^x-P_{\ell-1}^x)$:
\begin{align*}
    &\langle A \nabla_x (P_j-\widehat P_j) u, \nabla_x \HT \widehat{v}_j \rangle_{L^2(Q)^n}\\
    &\;= \sum_{\ell=0}^j \langle A \nabla_x \big((P_j^t-P_{j-\ell}^t)
    	\otimes (P_{\ell}^x - P_{\ell-1}^x)\big) u, \nabla_x \HT\widehat{v}_j\rangle_{L^2(Q)^n}\\
    &\;= \sum_{\ell=0}^j \langle A \nabla_x \big((P_j^t-P_{j-\ell}^t)
    	\otimes (P_{\ell}^x - P_{\ell-1}^x)\big) u, 
             \nabla_x \HT \big(\mathrm{Id}_t\otimes 
             	(P_{\ell}^x - P_{\ell-1}^x)\big) \widehat{v}_j\rangle_{L^2(Q)^n}\\
   &\;\le \sum_{\ell=0}^j \big\|\big((P_j^t-P_{j-\ell}^t)
    	\otimes (P_{\ell}^x - P_{\ell-1}^x)\big) u\big\|_{H^1(\domain) \otimes [H^{1/2}_{0,}(I)]'}\\
  &\hspace*{5cm}\times\big\|\big(\mathrm{Id}_t
  	\otimes (P_{\ell}^x-P_{\ell-1}^x)\big)\widehat{v}_j\big\|_{H^{1/2}_{0,}(I;H^1(\domain))}.
\end{align*}
Using finally the inverse inequality and the $L^2$-stability of 
the Ritz projections $P_\ell^x$, we derive
\[
  \big\|\big(\mathrm{Id}_t
  	\otimes (P_{\ell}^x-P_{\ell-1}^x) \big) \widehat{v}_j \big\|_{H^{1/2}_{0,}(I;H^1(\domain))}
	\le C (h_\ell^x)^{-1} \big\|\widehat{v}_j\big\|_{H^{1/2}_{0,}(I;L^2(\domain))},
\]
with a constant $C>0$, which proves the claim.
\end{proof}

The error bound in Theorem~\ref{thm:SpxtErrBd} implies convergence rate
bounds for solutions of sufficient spatio-temporal regularity. However, 
compared to the convergence rate in the full tensor-product space, 
we only obtain a reduced convergence rate in the spatial variable, 
even for regular solutions when we choose $P_j^t = P_j^{1/2,t}$ as
defined in \eqref{H12-Projektion}. Namely, for the single terms 
in \eqref{eq:SGerror} we have
\begin{gather*}
  \| u - \widehat P_j u \|_{H^{1/2}_{0,}(I; L^2(\domain))}\le 
  	C \sum_{\ell=0}^j (h_\ell^x)^{p_x+1} (h_{j-\ell}^t)^{d-1/2} \| u \|_{H^{d}(I; H^{p_x+1}(\domain))},\\
  \big\| (\mathrm{Id}_{xt} - P_j^t \otimes \mathrm{Id}_x) \mathcal L_x u \big\|_{[H^{1/2}_{,0}(I; L^2(\domain))]'}
  	\le C (h_j^t)^{d+1/2} \| u \|_{H^{d}(I; H^2(\domain))},
\end{gather*}
and
\begin{align*}
&\sum_{\ell=0}^j (h_\ell^x)^{-1} \big\|\big((P_j^t-P_{j-\ell}^t)
    	\otimes (P_{\ell}^x - P_{\ell-1}^x)\big) u\big\|_{H^1(\domain) \otimes [H^{1/2}_{0,}(I)]'}\\
 &\qquad\qquad\le C\sum_{\ell=0}^j (h_\ell^x)^{p_x-1}(h_{j-\ell}^t)^{d+1/2}
		\| u \|_{H^{d}(I; H^{p_x+1}(\domain))}.
\end{align*}
Altogether, this implies the estimate
\begin{equation}\label{eq:SGrate}
  \| u - \widehat u_j \|_{H^{1/2}_{0,}(I; L^2(\domain))} \le C 
	\max\big\{(h_j^x)^{p_x-1},(h_j^t)^{d-1/2}\big\}\| u \|_{H^{d}(I; H^{p_x+1}(\domain))}.
\end{equation}
This rate, however, is not observed in our numerical experiments:
for smooth solutions, essentially the same rate as obtained 
with the full tensor-product approximation space is observed. 

\begin{remark}
A possible explanation for this observation could be that the expression 
$\big(\mathrm{Id}_t\otimes (P_{\ell}^x-P_{\ell-1}^x)\big)\widehat{v}_j$
contains for $\widehat{v}_j\in\widehat{S}_j^{x,t}$ only contributions 
from the tensor-product spaces $S_k^x \otimes S_{j-k}^t$ for $k<\ell$.
Hence, if $P_\ell^t$ is chosen to be the projection $P_\ell^t\colon \, 
L^2(I) \to S_\ell^t$ defined by 
\begin{equation} \label{L2-Projektion}
  \forall z_\ell \in S_\ell^t : \quad  \langle P_\ell^t v, 
  	\HT z_\ell \rangle_{L^2(I)} = \langle v, \HT z_\ell \rangle_{L^2(I)} 
\end{equation}
for a given function $v \in L^2(I)$, then the term $\langle A \nabla_x 
(P_j-\widehat P_j) u, \nabla_x \HT \widehat{v}_j \rangle_{L^2(Q)^n}$
vanishes in the above proof. This would lead to the error estimate
\begin{multline*}
 \| u - \widehat u_j \|_{H^{1/2}_{0,}(I; L^2(\domain))}
  \le 
  2\| u - \widehat P_j u \|_{H^{1/2}_{0,}(I; L^2(\domain))}  \\
  + \big\| (\mathrm{Id}_{xt} - P_j^t \otimes \mathrm{Id}_x) \mathcal L_x u \big\|_{[H^{1/2}_{,0}(I; L^2(\domain))]'}.
\end{multline*}
Stability and consistency bounds for the projection 
\eqref{L2-Projektion} are delicate and topic of current investigation
\cite{LoescherSteinbachZankHT2024}.
\end{remark}

\begin{corollary}\label{cor:SpxtCnvRt}
Assume the sparse space-time Galerkin solution $\widehat{u}_j$
in \eqref{Sparse:WaermeFEM} has been obtained with temporal 
biorthogonal spline-wavelets of order $d\geq 2$ in $I$ and with 
Lagrangian finite elements of degree $p_x \geq 1$ in $\domain$.
Assume also that the exact solution $u$ of the variational IBVP 
\eqref{eq:IBVPVar} admits the regularity $u\in H^{d}(I;H^{1+p_x}
(\domain)) \simeq H^{1+p_x}(\domain) \otimes H^{d}(I)$.

Then, there holds the asymptotic error bound \eqref{eq:SGrate}.
Further, subject to the wavelet matrix compression of the temporal
stiffness and mass-matrices in Section~\ref{sec:wavelets} 
according to the specifications in Section~\ref{sec:MatCmpr},
the numerical solution $\widehat u_j$ can be computed in 
essentially\footnote{Up to an $O(j)$ factor.} 
$O(2^{n j})$ operations and memory.
\end{corollary}

\subsection{Iterative solution}
\label{sec:ItSol}
We show the claim on solver complexity in Corollary~\ref{cor:SpxtCnvRt}
and detail the sparse space-time matrix-vector multiplication,
and the corresponding preconditioner.
For the solution of the linear system of equations which
arises from the discretization \eqref{Sparse:WaermeFEM} 
of the heat equation with respect 
to the sparse space-time tensor-product space $\widehat{S}^{x,t}_j$, 
we use preconditioned GMRES.
To this end, 
we require matrix-vector products of the form
\begin{equation}\label{eq:MV-product}
  {\bf v}_{j-\ell',\ell'} = \sum_{\ell=0}^j
  \big({\bf A}_{j-\ell,j-\ell'}^t\otimes {\bf M}_{\ell,\ell'}^x
  + {\bf M}_{j-\ell,j-\ell'}^t\otimes {\bf A}_{\ell,\ell'}^x\big){\bf u}_{j-\ell,\ell}
\end{equation}
for all $0\le\ell'\le j$.
Here, 
the matrices ${\bf A}_{j-\ell,j-\ell'}^t$ 
and ${\bf M}_{j-\ell,j-\ell'}^t$ are specific blocks of the wavelet 
representations of the temporal matrices and therefore available. 
In contrast, the matrices ${\bf A}_{\ell,\ell'}^x$ and ${\bf M}_{\ell,\ell'}^x$ 
correspond to spatial finite element mass and stiffness 
matrices with trial and test functions from different 
discretization levels when $\ell\not=\ell'$. 
Such matrices are usually not available in practice. 
Therefore, we replace them with square 
finite element matrices and apply inter-grid 
restrictions to derive the desired ones. 
This means we employ the identities
\begin{equation}\label{eq:FEM1}
  {\bf A}_{\ell,\ell'}^x = {\bf I}_{\ell'}^\ell {\bf A}_{\ell'}^x,\quad
  {\bf M}_{\ell,\ell'}^x = {\bf I}_{\ell'}^\ell {\bf M}_{\ell'}^x,
  	\quad \text{if $\ell<\ell'$}, 
\end{equation}
and
\begin{equation}\label{eq:FEM2}
  {\bf A}_{\ell,\ell'}^x = {\bf A}_{\ell}^x {\bf I}_{\ell'}^\ell,\quad
  {\bf M}_{\ell,\ell'}^x = {\bf M}_{\ell}^x {\bf I}_{\ell'}^\ell,
  	\quad \text{if $\ell>\ell'$}.
\end{equation}
In both cases, the application of ${\bf A}_{\ell,\ell'}^x$
and ${\bf M}_{\ell,\ell'}^x$, respectively, scales 
linearly in the number 
$\max\{N_\ell^x,N_{\ell'}^x\}$ 
of degrees of freedom.

Note that the multilevel preconditioner in \eqref{eq:BPX} 
needs also to be adapted to the present setting. 
Namely, 
restricting \eqref{eq:BPX} to the sparse tensor-product 
space yields the multilevel preconditioner
\[
 \sum_{\ell=0}^j \big({\bf I}_{j-\ell}^t\otimes {\bf I}_\ell^j\big) 
 \Big(\big({\bf A}_{j-\ell}^t+(h_\ell^x)^{-2}{\bf M}_{j-\ell}^t\big)^{-1}\otimes {\bf I}_\ell^x\Big)
 	\big({\bf I}_{j-\ell}^t\otimes {\bf I}_j^\ell\big).
\]

\subsection{Fast matrix-vector products}
\label{sec:FastMV}
We shall explain how to compute matrix-vector products
in \eqref{eq:MV-product} in an efficient way. To this end, we will 
exploit that, given matrices ${\bf B}\in\mathbb{R}^{k\times\ell}$,
${\bf X}\in\mathbb{R}^{\ell\times m}$, ${\bf A}\in\mathbb{R}^{n\times m}$,
and ${\bf Y}\in\mathbb{R}^{k\times n}$, there holds the identity
\begin{equation}	\label{eq:equivalence}
  	\operatorname{vec}({\bf Y}) 
		= ({\bf A}\otimes{\bf B})\operatorname{vec}({\bf X})
        \quad\Longleftrightarrow\quad
        {\bf BXA}^\intercal = {\bf Y},
\end{equation}
where the vectorization $\operatorname{vec}(\cdot)$ converts 
a matrix into a column vector.
We will use this equivalence to realize a fast matrix-vector 
multiplication. For the sake of simplicity in presentation, we assume 
that the vector 
$\widehat{\bf u}_j = [{\bf u}_{\ell,j-\ell}]_{0\le\ell\le j}$ 
is partitioned into blocks of coefficients
associated with the sparse 
tensor-product basis in $\widehat{S}_j^{x,t}$ 
(compare \eqref{eq:sparseTPbasis}) 
and is blockwise stored in matrix form, i.e.\ 
${\bf u}_{\ell,j-\ell}\in\mathbb{R}^{|\Delta_\ell^x|\times|\nabla_{j-\ell}^t|}$.
Recall that we have  (cp. Sec.~\ref{sec:Notat})
$|\Delta_\ell^x| = N_\ell^x$ and $|\nabla_{j-\ell}^t| \sim N_{j-\ell}^t$.

Computation of one matrix-vector multiplication \eqref{eq:equivalence}
requires to compute products of the form
\[
  \operatorname{vec}({\bf z}) 
   = 
    \big({\bf A}_{j-\ell',j-\ell}^t\otimes{\bf B}_{\ell',\ell}^x\big)
  	\operatorname{vec}({\bf u}_{\ell,j-\ell}),
		\quad 0\le \ell,\ell' \leq j.
\]
In view of \eqref{eq:equivalence}, this means that
\[
  {\bf z} = {\bf B}_{\ell',\ell}^x{\bf u}_{\ell,j-\ell}\big({\bf A}^t_{j-\ell',j-\ell}\big)^\intercal,
	\quad 0\le \ell,\ell' \leq j.
\]
To optimize the complexity bound, 
care has to be taken with respect to
the order in which one performs the multiplications. 
If the matrix-vector multiplication 
is performed in accordance with
\[
  {\bf y} = {\bf u}_{\ell,j-\ell}\big({\bf A}^t_{j-\ell',j-\ell}\big)^\intercal,\quad
  {\bf z} = {\bf B}_{\ell',\ell}^x{\bf y},\quad\text{if}\ |\nabla_{j-\ell'}^t|\cdot 
  |\Delta_\ell^x|\le |\nabla_{j-\ell}^t|\cdot |\Delta_{\ell'}^x|,
\]
and
\[
  {\bf y} = {\bf B}_{\ell',\ell}^x{\bf u}_{\ell,j-\ell},\quad
  {\bf z} = {\bf y}\big({\bf A}^t_{j-\ell',j-\ell}\big)^\intercal,
  \quad\text{if}\ |\nabla_{j-\ell'}^t|\cdot |\Delta_\ell^x|
  > |\nabla_{j-\ell}^t|\cdot |\Delta_{\ell'}^x|,
\]
then we require at most $\mathcal{O}(\max\{N_j^x,N_j^t\})$ 
operations since the matrices ${\bf B}_{\ell',\ell}^x$ 
and ${\bf A}^t_{j-\ell',j-\ell}$ have only $\mathcal{O}
(\max\{|\Delta_\ell^x|,|\Delta_{\ell'}^x|\})$ and 
$\mathcal{O}(|\nabla_{j-\ell}^t|,|\nabla_{j-\ell'}^t|\})$
nonzero coefficients, respectively. Note that the 
application of the identities \eqref{eq:FEM1} and 
\eqref{eq:FEM2} do not change this complexity 
bound but avoid the non-quadratic finite element 
matrices ${\bf B}_{\ell',\ell}^x$ when $\ell\not= \ell'$. 
In all, we thus arrive at an algorithm that computes all 
matrix-vector product of the form \eqref{eq:MV-product}
in essentially linear complexity. This is achieved via 
the so-called ``unidirectional principle'', specifically 
Algorithm UNIDIRML in \cite{Zeiser:2011},
see also \cite{Bungartz.Griebel:2004}.

\subsection{Numerical results}
\label{sec:NumRes}
We recompute the example in Subsection~\ref{sec:Heat2D}
with respect to the sparse tensor-product space \eqref{Sparse:WaermeFEM}
instead of the full tensor-product space \eqref{FEMxt:WaermeFEM}. 
The results are displayed in Table \ref{tab:precond} 
in the columns labeled ``sparse tensor-product space''. 
Using the sparse tensor-product space results in an 
approximation accuracy that is basically the same 
as furnished with the full tensor-product space. 
Whereas, the number of 
iterations of the iterative solver is about a factor two larger 
in comparison with the full tensor-product space, compare
the 3rd and 5th column in Table \ref{tab:precond}. For example,
for the finest resolution, where $N_j^x = 1050625$ and
$N_j^t = 1024$, 
we have only about $12.5$ million degrees 
of freedom in the sparse tensor-product space instead 
of about one billion degrees of freedom in the full 
tensor-product space. 
Since a single iteration of the GMRES solver 
is therefore drastically cheaper due to the much smaller 
amount of degrees of freedom, the computing time and
memory requirement for the iterative solution is 
considerably smaller.

\section{Conclusion and further directions}
\label{sec:conclusio}
We proposed and analyzed a class of variational space-time discretizations
for initial-boundary value problems of linear, parabolic evolution equations, 
with finite time horizon $0<T<\infty$.
As in \cite{PSZ23,SteinbachZankETNA2020},
the space-time Galerkin 
algorithm is based on combining the classical variational formulation
in the spatial variable, combined with a new, $H^{1/2}_{0,}(0,T)$-coercive variational
formulation of the first-order evolution operator $\partial_t$. 
Whereas in \cite{PSZ23}, we leveraged the \emph{analytic semigroup property}
of the parabolic solution operator by exponentially convergent temporal 
Petrov--Galerkin discretizations, we focus in the present work on 
settings with solutions of low temporal regularity. 
These typically
arise in pathwise discretizations of linear, parabolic stochastic PDEs 
which are driven by noises of low pathwise temporal regularity,
such as cylindrical Wiener processes. 
For such PDEs, low temporal regularity
obstructs exponential convergence of $hp$-time discretizations of \cite{PSZ23}.
Time-adaptive schemes are likewise not offering substantial savings, 
as typically the temporal singular support of pathwise solutions is dense in $(0,T)$.
Accordingly, low-order, non-adaptive discretizations in $(0,T)$ afford optimal
convergence rates. 
Furthermore, 
\emph{multilevel space-time discretization} 
is desirable in connection with
so-called multilevel Monte-Carlo discretizations of such stochastic PDEs.

Stability of the continuous
variational formulation in the case of a finite time horizon $T$
and 
of its Galerkin discretization is based on an explicit, linear and nonlocal 
temporal operator, a modified Hilbert transform,
introduced first in \cite{SteinbachZankETNA2020}. 
Upon (Petrov--)Galerkin discretization with finite-dimensional 
spaces of Lagrangian finite element functions 
of degree $p_t\geq 1$ and of dimension $N_t$ in $(0,T)$, 
the evolution operator induces fully populated temporal matrices, i.e.\ 
a complexity $O(N_t^2)$ of memory and work for the time-derivative.
Similar to \cite{RStvWJW22}, where stability was ensured by
an adaptive wavelet algorithm,
we use suitable biorthogonal spline-wavelet bases in $(0,T)$
and standard Lagrangian finite element spaces in the physical domain $\domain$.
We proved that the densely populated 
matrices resulting from non-adaptive (Petrov--)Galerkin discretization
are optimally compressible, i.e.\ to $O(N_t)$ many nonvanishing entries.
With the aid of the exponential convergence quadrature schemes 
from \cite{ZankIntegral2023},
can be approximated computationally in $O(N_t)$ work and memory,
while maintaining the asymptotic discretization error bounds. 
Location of these $O(N_t)$ many nonvanishing matrix entries are known
in advance and the compressed temporal Galerkin 
matrix corresponding to $\HT$ can be precomputed in $O(N_t)$ work and memory.
We also obtained an $O(N_t)$ direct factorization by applying a nested
dissection strategy to the compressed wavelet Petrov--Galerkin matrix
resulting from the time-discretization.
We highlight that these findings hold for \emph{any} standard Galerkin discretization
of the spatial differential operator with respect to a single-scale basis. In particular, in space
dimension $n=2,3$, the number of spatial degrees of freedom $N_x$ is considerably
larger than $O(N_t)$, rendering the computational overhead 
for the temporal matrices negligible.

The multiresolution temporal (Petrov--)Galerkin discretization 
is combined via tensorization
with a multilevel finite element Galerkin discretization in the spatial domain $\domain$.
Here, standard, first-order Lagrangian finite element spaces on 
a sequence of nested, regular partitions of the physical domain $\domain$,
with $O(N_x)$ many degrees of freedom can be used; 
in particular, the construction of multiresolution analyses in the 
spatial domain $\Omega$ is not necessary.

We showed that
the nested nature of both, the spatial and the temporal subspaces used
in the Galerkin discretizations allows for straightforward 
realization of \emph{sparse space-time tensor-product} subspaces 
for the discretization of the evolution problem. 
The resulting sparse space-time tensor-product Galerkin discretizations
afford optimal approximation rates in $Q=\domain \times (0,T)$ while 
only using $O(N_x)$ work and memory, i.e.~the complexity of one space-time
solve reduces to work and memory of one multilevel solve of one spatial elliptic problem.

We proposed a new, iterative solver for the very large, linear systems
resulting from the sparse tensor-product discretizations.
It is based on standard, BPX-type multilevel iterative solvers 
combined with the nested dissection $LU$ factorizations
of the evolution operator, and we found that this results in $O(N_x)$ 
complexity solution of the fully discrete, sparse space-time (Petrov--)Galerkin
approximation.

Numerical experiments for first-order discretizations in space and time
confirmed and illustrated the above findings.

We indicate several directions for further work: 
in our consistency and discretization
error bounds, we assumed maximal spatial and temporal regularity of solutions.
While this regularity is available under conditions (such as 
compatible initial- and boundary data,
sufficiently smooth forcing, convex domains), 
the sparse space-time discretization and 
the multilevel solution algorithm adapt straightforwardly to more general situations:
for non-compatible initial data, the solutions are known to exhibit an 
\emph{algebraic temporal singularity as $t\downarrow 0$}, 
and for nonconvex spatial domains $\domain$ solutions will exhibit corner singularities
in space dimension $d=2$ and corner-edge singularities in addition in 
polyhedral $\domain\subset\mathbb{R}^3$, causing reduced convergence rates
of Galerkin discretizations on (quasi-)uniform meshes. 
Remedies restoring the full
convergence rates afforded by Lagrangian finite element methods 
in $\domain$ consist of graded or bisection-tree meshes with corner 
(resp.\ corner-edge) refinement (see, e.g., \cite[Section~5.3]{PSZ23}), 
and by 
local addition of wavelet-coefficients in $(0,T)$ near $t=0$. 
In both cases, extensions of the presently proposed, 
fast elliptic multilevel solution algorithms are available (e.g.\ \cite{BPXgraded}). 
Combining tensorizations of these with the presently proposed, wavelet-based 
time discretization would yield a performance comparable to the algorithms proposed here.

\appendix
\section{Piecewise linear wavelets on the interval}
\label{sec:appendix}
Let $\{\phi_{j,k}\}_{k=0}^{2^j+1}$ denote the standard
piecewise linear hat functions on the interval $[0,1]$ which
is supposed to be subdivided into $2^j$ equidistant subintervals 
of length $2^{-j}$. 
Then, the piecewise linear wavelets on level $j$ 
with zero Dirichlet boundary conditions at $x=0$ are given as follows.
\subsection{Two vanishing moments}
\begin{itemize}
\item Left boundary wavelet: 
\[
\psi_{j,1} = \frac{5}{8}\phi_{j,1}-\frac{3}{4}\phi_{j,2}
-\frac{1}{4}\phi_{j,3}+\frac{1}{4}\phi_{j,4}+\frac{1}{8}\phi_{j,5}
\]
\item 
Stationary wavelets ($k=2,\ldots,2^{j-1}-1$):
\[
\psi_{j,k} = -\frac{1}{8}\phi_{j,2k-3}-\frac{1}{4}\phi_{j,2k-2}+\frac{3}{4}\phi_{j,2k-1}
-\frac{1}{4}\phi_{j,2k}-\frac{1}{8}\phi_{j,2k+1}
\]
\item 
Right boundary wavelet:
\[
  \psi_{j,2^{j-1}} = -\frac{1}{16}\phi_{j,2^j-2}-\frac{1}{8}\phi_{j,2^j-1}
 +\frac{9}{16}\phi_{j,2^j}-\frac{3}{4}\phi_{j,2^j+1}
\]
\end{itemize}

\subsection{Four vanishing moments}
\begin{itemize}
\item
Left boundary wavelets:
\begin{align*}
\psi_{j,1} &= \frac{63}{128}\phi_{j,1}-\frac{65}{64}\phi_{j,2}
-\frac{1}{16}\phi_{j,3}+\frac{57}{64}\phi_{j,4}+\frac{13}{64}\phi_{j,5}\\
&\qquad-\frac{31}{64}\phi_{j,6}-\frac{3}{16}\phi_{j,7}
+\frac{7}{64}\phi_{j,8}+\frac{7}{128}\phi_{j,9}\\
\psi_{j,2} &= -\frac{7}{128}\phi_{j,1}-\frac{7}{64}\phi_{j,2}
+\frac{21}{32}\phi_{j,3}-\frac{37}{64}\phi_{j,4}-\frac{11}{64}\phi_{j,5}\\
&\qquad+\frac{15}{64}\phi_{j,6}+\frac{3}{32}\phi_{j,7}
-\frac{3}{64}\phi_{j,8}-\frac{3}{128}\phi_{j,9}
\end{align*}
Stationary wavelets ($k=3,\ldots,2^{j-1}-2$):
\begin{align*}
\psi_{j,k} &= \frac{3}{128}\phi_{j,2k-5}+\frac{3}{64}\phi_{j,2k-4}
-\frac{1}{8}\phi_{j,2k-3}-\frac{19}{64}\phi_{j,2k-2}
+\frac{45}{64}\phi_{j,2k-1}\\
&\qquad-\frac{19}{64}\phi_{j,2k}-\frac{1}{8}\phi_{j,2k+1}
+\frac{3}{64}\phi_{j,2k+2}+\frac{3}{128}\phi_{j,2k+3}
\end{align*}
\item
Right boundary wavelets:
\begin{align*}
  \psi_{j,2^{j-1}-1} &= \frac{9}{512}\phi_{j,2^j-6}+\frac{9}{256}\phi_{j,2^j-5}
  -\frac{53}{512}\phi_{j,2^j-4}-\frac{31}{128}\phi_{j,2^j-3}\\
  &\qquad+\frac{345}{512}\phi_{j,2^j-2}-\frac{105}{256}\phi_{j,2^j-1}
 -\frac{45}{512}\phi_{j,2^j}+\frac{15}{64}\phi_{j,2^j+1}\\
  \psi_{j,2^{j-1}} &= -\frac{5}{512}\phi_{j,2^j-6}-\frac{5}{256}\phi_{j,2^j-5}
  +\frac{67}{1536}\phi_{j,2^j-4}+\frac{41}{384}\phi_{j,2^j-3}\\
  &\qquad-\frac{53}{512}\phi_{j,2^j-2}-\frac{241}{768}\phi_{j,2^j-1}
 +\frac{875}{1536}\phi_{j,2^j}-\frac{35}{64}\phi_{j,2^j+1}
\end{align*} 
\end{itemize}

\section*{Acknowledgments}
\noindent
M.~Zank has been partially funded by the Austrian Science Fund (FWF) through
project P~33477. Part of the work was done when M.~Zank was a NAWI Graz
PostDoc Fellow at the Institute of Applied Mathematics, TU Graz.
M.~Zank acknowledges NAWI Graz for the financial support.

\bibliographystyle{amsplain}
\bibliography{lit.bib}{}

\providecommand{\bysame}{\leavevmode\hbox to3em{\hrulefill}\thinspace}
\providecommand{\MR}{\relax\ifhmode\unskip\space\fi MR }
\providecommand{\MRhref}[2]{%
  \href{http://www.ams.org/mathscinet-getitem?mr=#1}{#2}
}
\providecommand{\href}[2]{#2}
\begin{thebibliography}{10}

\bibitem{Andreev2013}
R.~Andreev, \emph{Stability of sparse space-time finite element discretizations
  of linear parabolic evolution equations}, IMA J. Numer. Anal. \textbf{33}
  (2013), no.~1, 242--260. \MR{3020957}

\bibitem{MR3449910}
\bysame, \emph{Wavelet-in-time multigrid-in-space preconditioning of parabolic
  evolution equations}, SIAM J. Sci. Comput. \textbf{38} (2016), no.~1,
  A216--A242.

\bibitem{BMPS21}
P.~Bansal, A.~Moiola, I.~Perugia, and C.~Schwab, \emph{Space-time discontinuous
  {G}alerkin approximation of acoustic waves with point singularities}, IMA J.
  Numer. Anal. \textbf{41} (2021), no.~3, 2056--2109.

\bibitem{BPX}
J.~Bramble, J.~Pasciak, and J.~Xu, \emph{{P}arallel multilevel
  preconditioners}, Math. Comput. \textbf{55} (1990), 1--22.

\bibitem{Bungartz.Griebel:2004}
H.-J. Bungartz and M.~Griebel, \emph{Sparse grids}, Acta Numer. \textbf{13}
  (2004), 1--123.

\bibitem{BPXgraded}
L.~Chen, R.H. Nochetto, and J.~Xu, \emph{Optimal multilevel methods for graded
  bisection grids}, Numer. Math. \textbf{120} (2012), no.~1, 1--34.

\bibitem{CDF}
A.~Cohen, I.~Daubechies, and J.-C. Feauveau, \emph{{B}iorthogonal bases of
  compactly supported wavelets}, Pure Appl. Math. \textbf{45} (1992), 485--560.

\bibitem{DA1}
W.~Dahmen, \emph{{W}avelet and multiscale methods for operator equations}, Acta
  Numer. \textbf{6} (1997), 55--228.

\bibitem{DHS1}
W.~Dahmen, H.~Harbrecht, and R.~Schneider, \emph{Compression techniques for
  boundary integral equations --- asymptotically optimal complexity estimates},
  SIAM J. Numer. Anal. \textbf{43} (2006), no.~6, 2251--2271.

\bibitem{DK92}
W.~Dahmen and A.~Kunoth, \emph{Multilevel preconditioning}, Numer. Math.
  \textbf{63} (1992), no.~3, 315--344.

\bibitem{DKU}
W.~Dahmen, A.~Kunoth, and K.~Urban, \emph{{B}iorthogonal spline-wavelets on the
  interval -- stability and moment conditions}, Appl. Comp. Harm. Anal.
  \textbf{6} (1999), 259--302.

\bibitem{DPS4}
W.~Dahmen, S.~{Pr\"{o}{\ss}dorf}, and R.~Schneider, \emph{Multiscale methods
  for pseudo\-differential equations on smooth manifolds}, Proceedings of the
  International Conference on Wavelets: Theory, Algorithms, and Applications
  (C.K. Chui, L.~Montefusco, and L.~Puccio, eds.), 1995, pp.~385--424.

\bibitem{DS}
W.~Dahmen and R.~Schneider, \emph{{W}avelets with complementary boundary
  conditions. function spaces on the cube}, Results Math. \textbf{34} (1998),
  255–293.

\bibitem{DL92}
R.~Dautray and J.-L. Lions, \emph{Mathematical analysis and numerical methods
  for science and technology. {V}ol. 5}, Springer-Verlag, Berlin, 1992.

\bibitem{GGRSt2021}
Gregor Gantner and Rob Stevenson, \emph{Further results on a space-time {FOSLS}
  formulation of parabolic {PDE}s}, ESAIM Math. Model. Numer. Anal. \textbf{55}
  (2021), no.~1, 283--299. \MR{4216839}

\bibitem{Geo73}
A.~George, \emph{Nested dissection of a regular finite element mesh}, SIAM J.
  Numer. Anal. \textbf{2} (1973), no.~10, 345--363.

\bibitem{GH13}
M.~Griebel and H.~Harbrecht, \emph{On the construction of sparse tensor product
  spaces}, Math. Comput. \textbf{82} (2013), no.~282, 975–994.

\bibitem{Griebel95}
M.~Griebel and P.~Oswald, \emph{On the abstract theory of additive and
  multiplicative {S}chwarz algorithms}, Numer. Math. \textbf{70} (1995), no.~2,
  163--180.

\bibitem{Gunzburger}
M.D. Gunzburger and A.~Kunoth, \emph{Space-time adaptive wavelet methods for
  control problems constrained by parabolic evolution equations}, SIAM J.
  Contr. Optim. \textbf{49} (2011), no.~3, 1150--1170.

\bibitem{Hack1}
W.~Hackbusch, \emph{A sparse matrix arithmetic based on
  {$\mathcal{H}$}-matrices. {P}art {I}: {I}ntroduction to
  {$\mathcal{H}$}-matrices}, Computing \textbf{62} (1999), no.~2, 89--108.

\bibitem{Hack2}
W.~Hackbusch and B.N. Khoromskij, \emph{A sparse {$\mathcal{H}$}-matrix
  arithmetic. {G}eneral complexity estimates}, J. Comput. Appl. Math.
  \textbf{125} (2000), no.~1–2, 479--501.

\bibitem{HM21}
H.~Harbrecht and M.D. Multerer, \emph{A fast direct solver for nonlocal
  operators in wavelet coordinates}, J. Comput. Phys. \textbf{428} (2021),
  110056.

\bibitem{HSS08}
H.~Harbrecht, R.~Schneider, and C.~Schwab, \emph{Multilevel frames for sparse
  tensor product spaces}, Numer. Math. \textbf{110} (2008), 199--220.

\bibitem{HauserZankMaxwell2024}
J.I.M. Hauser and M.~Zank, \emph{Numerical study of conforming space-time
  methods for {M}axwell's equations}, Numer. Methods Partial Differ. Equ.
  \textbf{40} (2024), no.~2, e23070.

\bibitem{Jaff}
S.\ Jaffard, \emph{{W}avelet methods for fast resolution of elliptic
  equations}, SIAM J.\ Numer.\ Anal. \textbf{29} (1992), 965--986.

\bibitem{LZ21}
U.~Langer and M.~Zank, \emph{Efficient direct space-time finite element solvers
  for parabolic initial-boundary value problems in anisotropic sobolev spaces},
  SIAM J. Sci. Comput. \textbf{4} (2021), no.~43, A2714--A2736.

\bibitem{LM1}
J.-L. Lions and E.~Magenes, \emph{Non-homogeneous boundary value problems and
  applications. {V}olume {I}.}, Die Grundlehren der mathematischen
  Wissenschaften in Einzeldarstellungen, vol. 181, Springer, New
  York-Heidelberg, 1972.

\bibitem{LRT79}
R.J. Lipton, D.J. Rose, and R.E. Tarjan, \emph{{G}eneralized nested
  dissection}, SIAM J. Numer. Anal. \textbf{2} (1979), no.~16, 346--358.

\bibitem{LoescherSteinbachZankHT2024}
R.~L\"oscher, O.~Steinbach, and M.~Zank, \emph{On a modified {H}ilbert
  transformation, the discrete inf-sup condition, and error estimates}, Comput.
  Math. Appl. \textbf{171} (2024), 114--138.

\bibitem{PSZ23}
I.~Perugia, C.~Schwab, and M.~Zank, \emph{Exponential convergence of
  {$hp$}-time-stepping in space-time discretizations of parabolic {PDE}s},
  ESAIM Math. Model. Numer. Anal. \textbf{57} (2023), no.~1, 29--67.

\bibitem{Ryan2002}
R.A. Ryan, \emph{Introduction to tensor products of {B}anach spaces}, Springer
  Monographs in Mathematics, Springer-Verlag London, Ltd., London, 2002.

\bibitem{SS86}
Y.\ Saad and M.H.\ Schultz, \emph{{GMRES}: A generalized minimal residual
  algorithm for solving nonsymmetric linear systems}, SIAM J.\ Sci.\ Statist.\
  Comput. \textbf{7} (1986), no.~3, 856--869.

\bibitem{SchneiderBuch1998}
R.~Schneider, \emph{Multiskalen- und {W}avelet-{M}atrixkompression}, Advances
  in Numerical Mathematics, B.G. Teubner, Stuttgart, 1998.

\bibitem{SS2000}
D.~Sch\"{o}tzau and C.~Schwab, \emph{Time discretization of parabolic problems
  by the {$hp$}-version of the discontinuous {G}alerkin finite element method},
  SIAM J. Numer. Anal. \textbf{38} (2000), no.~3, 837--875.

\bibitem{ScStxtWav}
C.~Schwab and R.~Stevenson, \emph{Space-time adaptive wavelet methods for
  parabolic evolution problems}, Math. Comput. \textbf{78} (2009), no.~267,
  1293--1318.

\bibitem{ScStxtNSE}
\bysame, \emph{Fractional space-time variational formulations of ({N}avier-)
  {S}tokes equations}, SIAM J. Math. Anal. \textbf{49} (2017), no.~4,
  2442--2467.

\bibitem{SteinbachZankETNA2020}
O.~Steinbach and M.~Zank, \emph{Coercive space-time finite element methods for
  initial boundary value problems}, Electron. Trans. Numer. Anal. \textbf{52}
  (2020), 154--194.

\bibitem{SteinbachZankJNUM2021}
\bysame, \emph{A note on the efficient evaluation of a modified {H}ilbert
  transformation}, J. Numer. Math. \textbf{29} (2021), no.~1, 47--61.

\bibitem{RSt2003}
R.~Stevenson, \emph{On the compressibility operators in wavelet coordinates},
  SIAM J. Math. Anal. \textbf{35} (2004), no.~5, 1110--1132.

\bibitem{RStvWJW22}
R.~Stevenson, R.~van Veneti\"{e}, and J.~Westerdiep, \emph{A wavelet-in-time,
  finite element-in-space adaptive method for parabolic evolution equations},
  Adv. Comput. Math. \textbf{48} (2022), no.~3, Paper No. 17, 43.

\bibitem{Thomee2nd}
V.~Thom\'{e}e, \emph{Galerkin finite element methods for parabolic problems},
  2nd ed., Springer Series in Computational Mathematics, vol.~25, Springer,
  Berlin, 2006.

\bibitem{ZankDissBuch2020}
M.~Zank, \emph{Inf-sup stable space-time methods for time-dependent partial
  differential equations}, Monographic Series TU Graz: Computation in
  Engineering and Science, vol.~36, TU Graz, Austria, 2020.

\bibitem{ZankExact2021}
\bysame, \emph{An exact realization of a modified {H}ilbert transformation for
  space-time methods for parabolic evolution equations}, Comput. Methods Appl.
  Math. \textbf{21} (2021), no.~2, 479--496.

\bibitem{ZankIntegral2023}
\bysame, \emph{Integral representations and quadrature schemes for the modified
  {H}ilbert transformation}, Comput. Methods Appl. Math. \textbf{23} (2023),
  no.~2, 473--489.

\bibitem{Zeiser:2011}
A.~Zeiser, \emph{Fast matrix-vector multiplication in the sparse-grid
  {G}alerkin method}, J. Sci. Comput. \textbf{47} (2011), no.~3, 328--346.

\end{thebibliography}

\end{document}